\documentclass[10pt,a4paper]{amsart}
\usepackage{amsmath,amsfonts,amsthm,amssymb,graphicx,mathrsfs,mathtools}
\usepackage[stretch=10]{microtype}
\usepackage[unicode]{hyperref}
\usepackage{xcolor}
\definecolor{dark-red}{rgb}{0.4,0.15,0.15}
\definecolor{dark-blue}{rgb}{0.15,0.15,0.4}
\definecolor{medium-blue}{rgb}{0,0,0.5}
\hypersetup{
	pdftitle={Equidistribution in Shrinking Sets and \texorpdfstring{$L^4$}{L\9040\164}-Norm Bounds for Automorphic Forms},
	pdfauthor={Peter Humphries},
	pdfnewwindow=true,   
	colorlinks, linkcolor={dark-red},
	citecolor={dark-blue}, urlcolor={medium-blue}
}
\makeatletter
\newcommand*{\defeq}{\mathrel{\rlap{%
			\raisebox{0.3ex}{$\m@th\cdot$}}%
		\raisebox{-0.3ex}{$\m@th\cdot$}}%
	=}
\newcommand*{\eqdef}{\mathrel{=\llap{%
			\raisebox{0.3ex}{$\m@th\cdot$}}%
		\llap{\raisebox{-0.3ex}{$\m@th\cdot$}}}%
}
\makeatother
\allowdisplaybreaks
\newcommand\A{\mathbb{A}}

\newcommand\BB{\mathcal{B}}
\renewcommand\C{\mathbb{C}}
\newcommand\CC{\mathcal{C}}
\newcommand\dd{\mathfrak{d}}

\newcommand\e{\varepsilon}
\newcommand\EE{\mathcal{E}}
\newcommand\GL{\mathrm{GL}}
\newcommand\Gs{\mathscr{G}}
\newcommand\Hb{\mathbb{H}}
\newcommand\NN{\mathcal{N}}
\newcommand\OO{\mathcal{O}}
\newcommand\pp{\mathfrak{p}}
\newcommand\PGL{\mathrm{PGL}}
\newcommand\PSL{\mathrm{PSL}}
\newcommand\qq{\mathfrak{q}}
\newcommand\reg{\mathrm{reg}}
\newcommand\Q{\mathbb{Q}}
\newcommand\R{\mathbb{R}}
\newcommand\SL{\mathrm{SL}}
\newcommand\SO{\mathrm{SO}}
\newcommand\T{\mathbb{T}}
\newcommand\tors{\mathrm{tors}}
\newcommand\Z{\mathbb{Z}}
\DeclareMathOperator{\arsinh}{arsinh}
\DeclareMathOperator{\Cl}{Cl}
\DeclareMathOperator{\Gen}{Gen}

\DeclareMathOperator{\sym}{sym}
\DeclareMathOperator{\Var}{Var}
\DeclareMathOperator{\vol}{vol}
\numberwithin{equation}{section}
\newtheorem{theorem}[equation]{Theorem}
\newtheorem{conjecture}[equation]{Conjecture}
\newtheorem{corollary}[equation]{Corollary}
\newtheorem{lemma}[equation]{Lemma}
\newtheorem{proposition}[equation]{Proposition}
\newtheorem{question}[equation]{Question}
\theoremstyle{remark}
\newtheorem{remark}[equation]{Remark}
\begin{document}

\title[Equidistribution in Shrinking Sets and $L^4$-Norm Bounds]{Equidistribution in Shrinking Sets and $L^4$-Norm Bounds for Automorphic Forms}

\author{Peter Humphries}

%\date{\today}

\address{Department of Mathematics, University College London, Gower Street, London WC1E 6BT}

\email{\href{mailto:pclhumphries@gmail.com}{pclhumphries@gmail.com}}



\begin{abstract}
We study two closely related problems stemming from the random wave conjecture for Maa\ss{} forms. The first problem is bounding the $L^4$-norm of a Maa\ss{} form in the large eigenvalue limit; we complete the work of Spinu to show that the $L^4$-norm of an Eisenstein series $E(z,1/2+it_g)$ restricted to compact sets is bounded by $\sqrt{\log t_g}$. The second problem is quantum unique ergodicity in shrinking sets; we show that by averaging over the centre of hyperbolic balls in $\Gamma \backslash \Hb$, quantum unique ergodicity holds for almost every shrinking ball whose radius is larger than the Planck scale. This result is conditional on the generalised Lindel\"{o}f hypothesis for Hecke--Maa\ss{} eigenforms but is unconditional for Eisenstein series. We also show that equidistribution for Hecke--Maa\ss{} eigenforms need not hold at or below the Planck scale. Finally, we prove similar equidistribution results in shrinking sets for Heegner points and closed geodesics associated to ideal classes of quadratic fields.
\end{abstract}

\maketitle

%\tableofcontents

\section{Introduction}

\subsection{Randomness of Maa\ss{} Newforms}\label{randomnesssect}

\subsubsection{Random Wave Conjecture}

Let $\BB_0(\Gamma)$ denote the set of Hecke--Maa\ss{} eigenforms of weight zero and level $1$ on the modular surface $\Gamma \backslash \Hb$, where $\Gamma = \SL_2(\Z)$ and $\Hb$ denotes the upper half-plane; we normalise $g \in \BB_0(\Gamma)$ to be such that
\[\langle g, g \rangle \defeq \int_{\Gamma \backslash \Hb} |g(z)|^2 \, d\mu(z) = 1,\]
where $d\mu(z) = y^{-2} \, dx \, dy$. A well-known conjecture of Berry \cite{Ber} and Hejhal and Rackner \cite{HejRa} states that a Hecke--Maa\ss{} eigenform $g \in \BB_0(\Gamma)$ of large Laplacian eigenvalue $\lambda_g = 1/4 + t_g^2$ ought to behave like a random wave. Here by a random wave, we mean a function of the form
\[g_{\lambda}(z) = \sum_{\lambda \leq \lambda_f \leq \lambda + \eta(\lambda)} c_f f(z),\]
where $\eta(\lambda) \to \infty$ as $\lambda \to \infty$ and $\eta(\lambda) = o(\lambda)$, each $f$ is a normalised Hecke--Maa\ss{} eigenform, and the coefficients $c_f$ are independent Gaussian random variables of mean $0$ and variance $1$. These are a randomised model of eigenfunctions of the Laplacian in the large eigenvalue limit $\lambda \to \infty$, and it is easier to prove (almost surely) results for random waves than for true eigenfunctions.

For $\Gamma \backslash \Hb$, there are situations in which random waves do not behave precisely like Laplacian eigenfunctions: random waves satisfy $\sup_{z \in K} |g_{\lambda}(z)| \asymp_K \sqrt{\log \lambda}$ almost surely for every compact subset $K$, whereas Mili\'{c}evi\'{c} \cite[Theorem 1]{Mil} proved the existence of a dense subset of points $z \in \Gamma \backslash \Hb$ for which a subsequence of Hecke--Maa\ss{} eigenforms $g \in \BB_0(\Gamma)$ may be much larger. Nonetheless, it is conjectured that Laplacian eigenfunctions should, on the whole, be well-modelled by random waves. This (admittedly loosely defined) conjecture is known as the random wave conjecture.

In this paper, we study two aspects of this conjecture: bounds for the $L^4$-norm of an automorphic form, and quantum unique ergodicity in shrinking balls. The former is a special case of the Gaussian moments conjecture, while the latter is a refinement of quantum unique ergodicity.

\subsubsection{Gaussian Moments Conjecture}

A particular manifestation of the random wave conjecture states that the moments of a Hecke--Maa\ss{} eigenform $g \in \BB_0(\Gamma)$ should be identical to those of a Gaussian random variable in the large eigenvalue limit.

\begin{conjecture}[Gaussian Moments Conjecture]
Let $K$ be any fixed compact continuity set of $\Gamma \backslash \Hb$, so that the boundary of $K$ has $\mu$-measure zero, and let $g \in \BB_0(\Gamma)$ be a Hecke--Maa\ss{} eigenform normalised such that $\langle g, g \rangle = 1$. Then for every nonnegative integer $n$,
\begin{equation}\label{Kgnmoment}
\frac{1}{\Var_K(g)^{n/2} \vol(K)} \int_K g(z)^n \, d\mu(z)
\end{equation}
converges to
\[\frac{1}{\sqrt{2\pi}} \int_{-\infty}^{\infty} x^n e^{-\frac{x^2}{2}} \, dx = \begin{dcases*}
\displaystyle \frac{2^{n/2}}{\sqrt{\pi}} \Gamma\left(\frac{n + 1}{2}\right) & if $n$ is even,	\\
0 & if $n$ is odd,
\end{dcases*}\]
as $t_g$ tends to infinity. Here 
\[\Var_K(g) \defeq \frac{1}{\vol(K)} \int_K |g(z)|^2 \, d\mu(z).\]
\end{conjecture}

When $K$ is replaced by a noncompact set, the Gaussian moments conjecture ought not necessarily to hold for high moments. As explained in \cite[Section 4]{HeSt}, using a heuristic appearing in \cite[Section 7]{Hej}, the transition range of the Whittaker function leads to a ``tidal pulse'' phenomenon near the cusp of $\Gamma \backslash \Hb$; when $K$ is replaced by $\Gamma \backslash \Hb$, so that $\Var_{\Gamma \backslash \Hb}(g) = \vol\left(\Gamma \backslash \Hb\right)^{-1}$, one can thereby show that there exists a subsequence of Hecke--Maa\ss{} eigenforms $g \in \BB_0(\Gamma)$ for which \eqref{Kgnmoment} grows like a power of $t_g$ whenever $n \geq 12$ is even. This is closely related to the fact that there exists a subsequence of Hecke--Maa\ss{} eigenforms for which
\[\|g\|_{\infty} \gg_{\e} t_g^{\frac{1}{6} - \e}.\]

Nonetheless, it is not unreasonable to conjecture that the Gaussian moments conjecture holds for smaller moments when $K$ is replaced by $\Gamma \backslash \Hb$. Indeed, the conjecture holds by definition for $n \in \{0,2\}$ and is easily shown to also be true when $n = 1$, as both sides vanish, while for $n = 3$, this can be shown to hold via the work of Watson \cite{Wat}.

\subsubsection{Quantum Unique Ergodicity}

Another manifestation of the randomness of Hecke--Maa\ss{} eigenforms is quantum unique ergodicity.

\begin{conjecture}[Quantum Unique Ergodicity in Configuration Space]
Let $g \in \BB_0(\Gamma)$ be a Hecke--Maa\ss{} eigenform normalised such that $\langle g, g \rangle = 1$. Then the probability measure $|g(z)|^2 \, d\mu(z)$ converges in distribution to the uniform probability measure on $\Gamma \backslash \Hb$ as $t_g$ tends to infinity, so that for every continuity set $B \subset \Gamma \backslash \Hb$,
\[\int_{B} |g(z)|^2 \, d\mu(z) = \frac{\vol(B)}{\vol\left(\Gamma \backslash \Hb\right)} + o_{B}(1)\]
as $t_g$ tends to infinity.
\end{conjecture}

By the Portmanteau theorem, this conjecture is equivalent to
\begin{equation}\label{QUEeq}
\int_{\Gamma \backslash \Hb} f(z) |g(z)|^2 \, d\mu(z) = \frac{1}{\vol\left(\Gamma \backslash \Hb\right)} \int_{\Gamma \backslash \Hb} f(z) \, d\mu(z) + o_{f}(1)
\end{equation}
for every bounded continuous function on $\Gamma \backslash \Hb$.

It behoves us to mention that there is a stronger formulation of quantum unique ergodicity, namely quantum unique ergodicity in phase space, which is the cosphere bundle $S^{\ast} \left(\Gamma \backslash \Hb\right) \cong \Gamma \backslash \SL_2(\R)$: not only should the sequence of probability measures $|g(z)|^2 \, d\mu(z)$ equidistribute on the configuration space $\Gamma \backslash \Hb$, but that a microlocal lift of these measures to Wigner distributions on phase space should equidistribute with respect to the Liouville measure.

Quantum unique ergodicity in phase space, and hence also in configuration space, is known to be true via the work of Lindenstrauss \cite{Lin} and Soundararajan \cite{Sou}. However, this proof does not quantify the rate of equidistribution; in particular, it does not give explicit rates of decay for the terms
\begin{equation}\label{basisQUE}
\int_{\Gamma \backslash \Hb} f(z) |g(z)|^2 \, d\mu(z)
\end{equation}
for fixed $f \in C_b(\Gamma \backslash \Hb)$ as $t_g$ tends to infinity. Watson \cite[Corollary 1]{Wat} has shown that optimal decay rates for these integrals follow directly from the generalised Lindel\"{o}f hypothesis.

The $n = 2$ case of the Gaussian moments conjecture for the set $K = \Gamma \backslash \Hb$ --- namely the $L^4$-norm of $g$ --- shares many similarities with quantum unique ergodicity in configuration space. In fact, it is extremely closely related to a more refined version of quantum unique ergodicity, namely equidistribution on shrinking sets.

\subsubsection{Randomness of Eisenstein Series}

The  Gaussian moments conjecture and quantum unique ergodicity ought to be true, once suitably modified, when $g(z) = E(z, 1/2 + it_g)$ is an Eisenstein series. Eisenstein series are not square-integrable, so one must use some sort of regularisation. One method is to use Zagier's regularisation of divergent integrals \cite{Zag}; another is to replace $E(z, 1/2 + it_g)$ with the truncated Eisenstein series $\Lambda^T E(z,1/2 + it_g)$ for some $T \geq 1$; this is defined for $\Re(s) > 1$ by
\[\Lambda^T E(z,s) \defeq E(z,s) - \sum_{\substack{\gamma \in \Gamma_{\infty} \backslash \Gamma \\ \Im(\gamma z) > T}} \left(\Im(\gamma z)^s + \frac{\Lambda(2 - 2s)}{\Lambda(2s)} \Im(\gamma z)^{1 - s}\right)\]
and extended by meromorphic continuation to the complex plane; here $\Lambda(s)$ denotes the completed Riemann zeta function.

For quantum unique ergodicity, we need not deal with the truncated version of the Eisenstein series provided that we take into account the growth of the $L^2$-norm of an Eisenstein series on compact sets.

\begin{theorem}[Luo--Sarnak {\cite[Theorem 1.1]{LS}}]\label{QUEeisen}
For any compact continuity set $K \subset \Gamma \backslash \Hb$ and for $g(z) = E\left(z, 1/2 + it_g\right)$,
\[\int_K |g(z)|^2 \, d\mu(z) = \frac{\log \left(\frac{1}{4} + t_g^2\right) \vol(K)}{\vol\left(\Gamma \backslash \Hb\right)} + o_K\left(\log t_g\right)\]
as $t_g$ tends to infinity.
\end{theorem}

Since $K$ is compact, one can replace $g(z)$ with $\Lambda^T E\left(z, 1/2 + it_g\right)$ for some $T$ sufficiently large dependent on $K$. The presence of $\log (1/4 + t_g^2)$ essentially stems from the Maa\ss{}--Selberg relation; see \hyperref[constantcontrib]{Corollary \ref*{constantcontrib}}.

Quantum unique ergodicity in phase space is also known for Eisenstein series; this is a result of Jakobson \cite[Theorem 1]{Jak}.

\subsection{The \texorpdfstring{$L^4$}{L\9040\164}-Norm Problem}\label{L4normproblemsect}

The $L^4$-norm problem for a Hecke--Maa\ss{} eigenform $g$ is the second nontrivial case of the Gaussian moments conjecture.

\begin{conjecture}[$L^4$-Norm Problem]
Let $g \in \BB_0(\Gamma)$ be a Hecke--Maa\ss{} eigenform normalised such that $\langle g, g \rangle = 1$. As $t_g$ tends to infinity,
\[\int_{\Gamma \backslash \Hb} |g(z)|^4 \, d\mu(z) = \frac{3}{\vol\left(\Gamma \backslash \Hb\right)} + o(1).\]
\end{conjecture}

A similar statement can be formulated when $g$ is an Eisenstein series, though some care must be taken, since Eisenstein series are not square-integrable; see \cite{DK18}.

In general, an unconditional proof of the $L^4$-norm problem seems quite difficult. A weaker conjecture (see, for example, \cite[Conjecture 4]{Sar}) is that
\begin{equation}\label{L4epseq}
\|g\|_{L^4\left(\Gamma \backslash \Hb\right)}^4 \ll_{\e} t_g^{\e}.
\end{equation}
In certain special cases, this has been shown: when $g$ is a dihedral Maa\ss{} eigenform, this is a result of Luo \cite{Luo}, while when $g$ is a truncated Eisenstein series, this is a result of Spinu \cite{Spi} (with the implicit constant of course dependent on the truncation parameter $T$).

Buttcane and Khan \cite[Theorem 1.1]{BK17b} have recently given a proof, conditional on the generalised Lindel\"{o}f hypothesis, of the $L^4$-norm problem for a Hecke--Maa\ss{} eigenform $g \in \BB_0(\Gamma)$. Our first main result is to give an unconditional upper bound for the $L^4$-norm of a truncated Eisenstein series that is sharper than \eqref{L4epseq}.

\begin{theorem}\label{L4uncondthm}
Let $g(z) = \Lambda^T E\left(z, 1/2 + it_g\right)$. We have that
\[\|g\|_{L^4\left(\Gamma \backslash \Hb\right)}^4 \ll_T \left(\log t_g\right)^2.\]
\end{theorem}

Up to the implicit constant, \hyperref[L4uncondthm]{Theorem \ref*{L4uncondthm}} should be sharp, for the Maa\ss{}--Selberg relation implies that
\[\|g\|_{L^2\left(\Gamma \backslash \Hb\right)}^4 = \left(\log \left(\left(\frac{1}{4} + t_g^2\right) T^2\right) + O\left(\left(\log t_g\right)^{2/3} \left(\log \log t_g\right)^{1/3}\right)\right)^2.\]

\begin{remark}
\hyperref[L4uncondthm]{Theorem \ref*{L4uncondthm}} was previously claimed by Spinu \cite[Theorem 1.2]{Spi}, as was a proof of \eqref{L4epseq} for Hecke--Maa\ss{} cusp forms by Sarnak and Watson \cite[Theorem 3]{Sar}; in both cases, however, the proofs are incomplete, as we shall discuss further in \hyperref[SpinuSarnakremark]{Remark \ref*{SpinuSarnakremark}}.
\end{remark}

\begin{remark}
Djankovi\'{c} and Khan \cite{DK18} have recently reformulated the $L^4$-norm problem for Eisenstein series by studying a regularised fourth moment of an Eisenstein series in the sense of Zagier \cite{Zag}; cf.~\hyperref[WatsonIchinosect]{Section \ref*{WatsonIchinosect}}. This has the advantage that one ought to be able to prove an asymptotic for this regularised fourth moment, whereas \hyperref[L4uncondthm]{Theorem \ref*{L4uncondthm}} only provides an upper bound for the fourth moment of a truncated Eisenstein series.
\end{remark}

\subsection{Quantum Unique Ergodicity in Shrinking Sets}\label{Eqshrinkingsetsect}

A natural strengthening of quantum unique ergodicity is to determine whether equidistribution still occurs if we vary the set $B$ with $t_g$; in particular, if the size of $B$ shrinks as $t_g$ increases. This small scale equidistribution should be thought of as a reinterpretation of determining the rate of equidistribution, as opposed to determining explicit rates of decay for the terms in \eqref{basisQUE}. Proving equidistribution in shrinking sets has applications towards bounds for the $L^p$-norms and size of nodal domains of eigenfunctions of the Laplacian; see \cite{HezRi16}.

We denote by $B = B_R(w)$ the hyperbolic ball of radius $R$ centred at $w \in \Gamma \backslash \Hb$: its hyperbolic volume is
\[\vol\left(B_R\right) = 4\pi \sinh^2 \frac{R}{2},\]
which is independent of the centre $w$.

\begin{question}\label{shrinkingq}
Let $g \in \BB_0(\Gamma)$ be a Hecke--Maa\ss{} eigenform normalised such that $\langle g, g \rangle = 1$. For what conditions on $R$, with regards to $t_g$, is it still true that
\begin{equation}\label{shrinkingQUEeq}
\frac{1}{\vol\left(B_R\right)} \int_{B_R(w)} |g(z)|^2 \, d\mu(z) = \frac{1}{\vol\left(\Gamma \backslash \Hb\right)} + o_w(1)
\end{equation}
as $t_g$ tends to infinity?
\end{question}

In the general setting of negatively curved manifolds, this question has independently been answered by Han \cite[Theorem 1.5]{Han15} and Hezari and Rivi\`{e}re \cite[Proposition 2.1]{HezRi16} for a full density subsequence of Laplacian eigenfunctions with the radius $R$ shrinking at a rate $(\log \lambda_g)^{-\beta}$ for a particular range of $\beta > 0$ dependent on the manifold.

We should not expect equidistribution to hold when $R \ll t_g^{-1}$; indeed, Hejhal and Rackner \cite[Section 5]{HejRa}, writing $\Psi_n$ in place of $g$, $\lambda_n$ in place of $\lambda_g = 1/4 + t_g^2$, and $A$ in place of $R$, state that
\begin{quote}
\ldots in the physics literature, $c / \sqrt{\lambda_n}$ is commonly referred to as the de Broglie wavelength. At length scales below $c / \sqrt{\lambda_n}$, one expects the topography of $\Psi_n$ to look ``essentially sinusoidal'', that is, regular. It is only when $A$ is substantially bigger than the de Broglie wavelength that one stands any chance of seeing any type of Gaussian distribution.
\end{quote}
We confirm this statement by showing that if $R \ll_A t_g^{-1} (\log t_g)^A$ for any $A > 0$, then there exist infinitely many points $w \in \Gamma \backslash \Hb$ for which \eqref{shrinkingQUEeq} does \emph{not} hold, so that the sequence of probability measures $|g(z)|^2 \, d\mu(z)$ does not equidistribute on the shrinking balls of radius $t_g^{-1} (\log t_g)^A$ centred at these points. We think of $R \asymp t_g^{-1}$ as being the Planck scale, so that equidistribution need not occur within a logarithmic window of the Planck scale.

\begin{theorem}\label{Planckscalethm}
Let $g \in \BB_0(\Gamma)$ be a Hecke--Maa\ss{} eigenform normalised such that $\langle g, g \rangle = 1$. For every fixed Heegner point $w \in \Gamma \backslash \Hb$, we have that
\[\frac{1}{\vol\left(B_R\right)} \int_{B_R(w)} |g(z)|^2 \, d\mu(z) = \Omega\left(\exp\left(2\sqrt{\frac{\log t_g}{\log \log t_g}} \left(1 + O\left(\frac{\log \log \log t_g}{\log \log t_g}\right)\right)\right)\right)\]
for $R \ll_A t_g^{-1} (\log t_g)^A$ for any $A > 0$ as $t_g$ tends to infinity.
\end{theorem}

Nevertheless, we should expect equidistribution to occur at every scale larger than the Planck scale, namely $R \gg t_g^{-\delta}$ for any $\delta < 1$. Towards this, Young \cite{You16} has proved the following.

\begin{theorem}[{Young \cite[Proposition 1.5]{You16}}]
Let $g \in \BB_0(\Gamma)$ be a Hecke--Maa\ss{} eigenform normalised such that $\langle g, g \rangle = 1$. Assume the generalised Lindel\"{o}f hypothesis, and suppose that $R \asymp t_g^{-\delta}$ with $\delta < 1/3$. Then
\[\frac{1}{\vol\left(B_R\right)} \int_{B_R(w)} |g(z)|^2 \, d\mu(z) = \frac{1}{\vol\left(\Gamma \backslash \Hb\right)} + o_{w,\delta}(1)\]
for every fixed point $w \in \Gamma \backslash \Hb$.

Similarly, let $g(z) = E(z,1/2 + it_g)$, and suppose that $R \asymp t_g^{-\delta}$ with $\delta < 1/9$. Then unconditionally
\[\frac{1}{\log \left(\frac{1}{4} + t_g^2\right) \vol\left(B_R\right)} \int_{B_R(w)} |g(z)|^2 \, d\mu(z) = \frac{1}{\vol(\Gamma \backslash \Hb)} + o_{w,\delta}(1)\]
for every fixed point $w \in \Gamma \backslash \Hb$.
\end{theorem}

In fact, with little work, we can improve the range in Young's result for Eisenstein series.

\begin{theorem}\label{improvedYoungEisenthm}
Let $g(z) = E(z,1/2 + it_g)$, and suppose that $R \asymp t_g^{-\delta}$ with $\delta < 1/6$. Then unconditionally
\[\frac{1}{\log \left(\frac{1}{4} + t_g^2\right) \vol\left(B_R\right)} \int_{B_R(w)} |g(z)|^2 \, d\mu(z) = \frac{1}{\vol(\Gamma \backslash \Hb)} + o_{w,\delta}(1)\]
for every fixed point $w \in \Gamma \backslash \Hb$.
\end{theorem}

A simpler version of \hyperref[shrinkingq]{Question \ref*{shrinkingq}} is to instead consider eigenfunctions of the Laplacian on the $d$-torus $\T^d$ for any $d \geq 2$. Hezari and Rivi\`{e}re \cite[Corollary 1.5]{HezRi17} give strong bounds for equidistribution in shrinking balls along a full density subsequence of eigenfunctions of the Laplacian on $\T^d$ with eigenvalue $\lambda$, namely equidistribution on all balls of radius $R \gg \lambda^{-\frac{1}{4(d + 1)}}$. Lester and Rudnick \cite[Theorem 1.1]{LR} improve this to $R \gg_{\e} \lambda^{-\frac{1}{2(d - 1)} + \e}$. Moreover, they prove \cite[Theorems 3.1 and 4.1]{LR} that this is essentially sharp, in that there exists a subsequence of eigenfunctions for which equidistribution does not occur on shrinking balls of radius $R \ll_{\e} \lambda^{-\frac{1}{2(d - 1)} - \e}$. For $d = 2$, Granville and Wigman \cite[Corollary 3.2]{GW} have subsequently sharpened Lester and Rudnick's results to show there exists $A > 0$ such that equidistribution may not occur on shrinking balls of radius $R \ll_A \lambda^{-1/2} (\log \lambda)^A$.

One can also reformulate \hyperref[shrinkingq]{Question \ref*{shrinkingq}} probabilistically by asking for which scales equidistribution holds almost surely with respect to a random eigenbasis of Laplacian eigenfunctions; positive results towards this question appear in the work of Han \cite{Han17} and Han and Tacy \cite{HT}.

We study a related question: instead of demanding that equidistribution hold in shrinking balls of radius $R > 0$ centred at $w$ for every point $w \in \Gamma \backslash \Hb$, we relax this requirement by instead asking whether equidistribution holds in shrinking balls $B_R(w)$ for almost every $w \in \Gamma \backslash \Hb$.

\subsubsection{Conditional Results}

We are able to give a conditional proof of equidistribution in almost every shrinking ball when $g \in \BB_0(\Gamma)$ and $R \gg t_g^{-\delta}$ for any $0 < \delta < 1$, that is, at all scales above the Planck scale.

\begin{theorem}\label{Maassalmosteverythm}
Let $g \in \BB_0(\Gamma)$ be a Hecke--Maa\ss{} eigenform normalised such that $\langle g, g \rangle = 1$. Assume the generalised Lindel\"{o}f hypothesis, and suppose that $R \asymp t_g^{-\delta}$ for some $0 < \delta < 1$. Then for any $c \gg_{\e} t_g^{- \frac{1 - \delta}{2} + \e}$,
\[\vol\left(\left\{w \in \Gamma \backslash \Hb : \left|\frac{1}{\vol\left(B_R\right)} \int_{B_R(w)} |g(z)|^2 \, d\mu(z) - \frac{1}{\vol\left(\Gamma \backslash \Hb\right)}\right| > c\right\}\right)\]
converges to zero as $t_g$ tends to infinity.
\end{theorem}

\subsubsection{Unconditional Results}

Proving unconditional results seems to be much more difficult. Nevertheless, we are able to do so when $g(z) = E\left(z, 1/2 + it_g\right)$ is an Eisenstein series.

\begin{theorem}\label{Eisenalmosteverythm}
Let $g(z) = E\left(z, 1/2 + it_g\right)$. Suppose that $R \asymp t_g^{-\delta}$ for some $0 < \delta < 1$. Then for any $c \gg_{\e} t_g^{-\min\left\{\frac{5}{14} (1 - \delta), 2\delta, \frac{1}{12}\right\} + \e}$,
\[\vol\left(\left\{w \in \Gamma \backslash \Hb : \left|\frac{1}{\vol\left(B_R\right)} \int_{B_R(w)} |g(z)|^2 \, d\mu(z) - D(g;w)\right| > c\right\}\right)\]
converges to zero as $t_g$ tends to infinity, where $D(g;w)$ is given by \eqref{D(g;w)}.
\end{theorem}

This result is consistent with \hyperref[QUEeisen]{Theorem \ref*{QUEeisen}} due to the following.

\begin{lemma}\label{D(g;w)compactlemma}
In any compact subset $K$ of $\Gamma \backslash \Hb$, we have that for all $w \in K$,
\[D(g;w) = \frac{\log \left(\frac{1}{4} + t_g^2\right)}{\vol\left(\Gamma \backslash \Hb\right)} + O_K\left(\left(\log t_g\right)^{2/3} \left(\log \log t_g\right)^{1/3}\right).\]
\end{lemma}

In particular, we may rephrase \hyperref[Eisenalmosteverythm]{Theorem \ref*{Eisenalmosteverythm}} in the following way.

\begin{corollary}
Let $g(z) = E\left(z, 1/2 + it_g\right)$, and let $K$ be a fixed compact subset of $\Gamma \backslash \Hb$. Suppose that $R \gg_{\e} t_g^{-1 + \e}$. Then for any fixed $c  > 0$,
\[\vol\left(\left\{w \in K : \left|\frac{1}{\log \left(\frac{1}{4} + t_g^2\right) \vol\left(B_R\right)} \int_{B_R(w)} |g(z)|^2 \, d\mu(z) - \frac{1}{\vol\left(\Gamma \backslash \Hb\right)}\right| > c\right\}\right)\]
converges to zero as $t_g$ tends to infinity.
\end{corollary}

\subsection{Equidistribution of Geometric Invariants of Quadratic Fields in Shrinking Sets}
\label{GeomInvintrosect}

Finally, in \hyperref[section6]{Section \ref*{section6}}, we study a similar equidistribution problem in shrinking sets. Associated to each narrow ideal class $A$ of the narrow class group $\Cl_K^+$ of a quadratic number field $K = \Q(\sqrt{D})$ is a geometric invariant. For $D < 0$, this is a Heegner point $z_A$, while for $D > 0$, this is a closed geodesic $\CC_A$ or a hyperbolic orbifold $\Gamma_A \backslash \NN_A$ having this closed geodesic as its boundary; we explain these geometric invariants in more detail in \hyperref[geominvsect]{Section \ref*{geominvsect}}.

For each fundamental discriminant $D$, we choose a genus $G_K \subset \Cl_K^+$ in the group of genera $\Gen_K = \Cl_K^+ / (\Cl_K^+)^2$, so that $G_K$ is a coset $A (\Cl_K^+)^2$ of narrow ideal classes in $\Cl_K^+$. We have that $\Gen_K \cong (\Z / 2\Z)^{\omega(|D|) - 1}$, where $\omega(|D|)$ is the number of distinct prime factors of $|D|$, so that $\# G_K = \# (\Cl_K^+)^2 = 2^{1 - \omega(|D|)} h_K^+$, where $h_K^+ \defeq \# \Cl_K^+$ denotes the narrow class number of $K$. Duke, Imamo\={g}lu, and T\'{o}th have proved the following equidistribution theorem.

\begin{theorem}[{\cite[Theorem 2]{DIT}}]
For every continuity set $B \subset \Gamma \backslash \Hb$,
\begin{align*}
\frac{\#\left\{A \in G_K : z_A \in B\right\}}{\# G_K} & = \frac{\vol(B)}{\vol\left(\Gamma \backslash \Hb\right)} + o_B(1)
\intertext{as $D \to -\infty$ through fundamental discriminants, and}
\frac{\sum_{A \in G_K} \ell\left(\CC_A \cap B\right)}{\sum_{A \in G_K} \ell\left(\CC_A\right)} & = \frac{\vol(B)}{\vol\left(\Gamma \backslash \Hb\right)} + o_B(1),	\\
\frac{\sum_{A \in G_K} \vol\left(\Gamma_A \backslash \NN_A \cap B\right)}{\sum_{A \in G_K} \vol\left(\Gamma_A \backslash \NN_A\right)} & = \frac{\vol(B)}{\vol\left(\Gamma \backslash \Hb\right)} + o_B(1)
\end{align*}
as $D \to \infty$ through fundamental discriminants, where $\ell(\CC_A) \defeq \int_{\CC_A} \, ds$, with $ds^2 = y^{-2} dx^2 + y^{-2} dy^2$.
\end{theorem}

If we sum over all genera, so that we are studying equidistribution associated to the full narrow class group, then this result is due to Duke \cite[Theorem 1]{Duk} for Heegner points and closed geodesics, while this result becomes trivial for hyperbolic orbifolds, for there is no error term whatsoever in this case. Moreover, the equidistribution of closed geodesics has a stronger realisation: instead of merely asking for the equidistribution of closed geodesics on $\Gamma \backslash \Hb$, we may lift these geodesics to phase space $S^{\ast} \left(\Gamma \backslash \Hb\right) \cong \Gamma \backslash \SL_2(\R)$ and demand equidistribution with respect to the Liouville measure. This has been proved by Chelluri \cite{Che}.

It is natural to ask whether equidistribution still occurs if $B$ shrinks as $|D|$ grows. Towards this, Young \cite{You17a} has proved the following.

\begin{theorem}[{Young \cite[Theorem 2.1]{You17a}}]
Fix $w \in \Gamma \backslash \Hb$, and suppose that $R \asymp (-D)^{-\delta}$. Unconditionally, as $D \to -\infty$ through odd fundamental discriminants,
\begin{equation}\label{YoungHeegnereq}
\frac{\#\left\{A \in \Cl_K : z_A \in B_R(w)\right\}}{\vol\left(B_R\right) h_K} = \frac{1}{\vol\left(\Gamma \backslash \Hb\right)} + o_{w,\delta}(1)
\end{equation}
for fixed $\delta < 1/24$, where $\Cl_K$ denotes the class group of $K$ and $h_K \defeq \# \Cl_K$ denotes the class number. Assuming the generalised Lindel\"{o}f hypothesis, \eqref{YoungHeegnereq} holds as $D \to -\infty$ through fundamental discriminants for fixed $\delta < 1/8$.
\end{theorem}

In fact, from the method of proof, it is clear that Young's theorem applies to genera mutatis mutandis, and proves equidistribution not only of Heegner points, but also of closed geodesics and hyperbolic orbifolds.

\begin{theorem}
Fix $w \in \Gamma \backslash \Hb$, and suppose that $R \asymp D^{-\delta}$. Unconditionally, as $D \to \infty$ through odd fundamental discriminants,
\begin{align*}
\frac{\sum_{A \in G_K} \ell\left(\CC_A \cap B_R(w)\right)}{\vol\left(B_R\right) \sum_{A \in G_K} \ell\left(\CC_A\right)} & = \frac{1}{\vol\left(\Gamma \backslash \Hb\right)} + o_{w,\delta}(1) \qquad \text{for $\delta < 1/18$,}	\\
\frac{\sum_{A \in G_K} \vol\left(\Gamma_A \backslash \NN_A \cap B_R(w)\right)}{\vol\left(B_R\right) \sum_{A \in G_K} \vol\left(\Gamma_A \backslash \NN_A\right)} & = \frac{1}{\vol\left(\Gamma \backslash \Hb\right)} + o_{w,\delta}(1) \qquad \text{for $\delta < 1/12$.}
\end{align*}
Assuming the generalised Lindel\"{o}f hypothesis, these hold as $D \to \infty$ through fundamental discriminants for $\delta < 1/6$ and $\delta < 1/4$ respectively.
\end{theorem}

Once again, we may weaken the demand that equidistribution hold in shrinking balls of radius $R > 0$ centred at $w$ for every point $w \in \Gamma \backslash \Hb$ and instead study whether equidistribution holds in shrinking balls $B_R(w)$ for almost every $w \in \Gamma \backslash \Hb$.

We prove the following conditional result.

\begin{theorem}\label{quadalmosteverycondthm}
Suppose that $R \asymp |D|^{-\delta}$. Assuming the generalised Lindel\"{o}f hypothesis, we have that for $0 < \delta < 1/4$ and $c \gg_{\e} (-D)^{-\frac{1}{2} \left(\frac{1}{4} - \delta\right) + \e}$,
\[\vol\left(\left\{w \in \Gamma \backslash \Hb : \left|\frac{\#\left\{A \in G_K : z_A \in B_R(w)\right\}}{\vol\left(B_R\right) \# G_K} - \frac{1}{\vol\left(\Gamma \backslash \Hb\right)}\right| > c\right\}\right)\]
converges to zero as $D \to -\infty$ along fundamental discriminants, while for $0 < \delta < 1/2$ and $c \gg_{\e} D^{-\frac{1}{2} \left(\frac{1}{2} - \delta\right) + \e}$,
\[\vol\left(\left\{w \in \Gamma \backslash \Hb : \left|\frac{\sum_{A \in G_K} \ell\left(\CC_A \cap B_R(w)\right)}{\vol\left(B_R\right) \sum_{A \in G_K} \ell\left(\CC_A\right)} - \frac{1}{\vol\left(\Gamma \backslash \Hb\right)}\right| > c\right\}\right)\]
converges to zero as $D \to \infty$ along fundamental discriminants.
\end{theorem}

Unconditionally, we obtain the following weaker results.

\begin{theorem}\label{quadalmosteveryuncondthm}
Suppose that $R \asymp |D|^{-\delta}$. We have that for $0 < \delta < 1/12$ and $c \gg_{\e} (-D)^{-\frac{1}{2} \left(\frac{1}{12} - \delta\right) + \e}$,
\[\vol\left(\left\{w \in \Gamma \backslash \Hb : \left|\frac{\#\left\{A \in G_K : z_A \in B_R(w)\right\}}{\vol\left(B_R\right) \# G_K} - \frac{1}{\vol\left(\Gamma \backslash \Hb\right)}\right| > c\right\}\right)\]
converges to zero as $D \to -\infty$ along odd fundamental discriminants, while for $0 < \delta < 1/6$ and $c \gg_{\e} D^{-\frac{1}{2} \left(\frac{1}{6} - \delta\right) + \e}$,
\[\vol\left(\left\{w \in \Gamma \backslash \Hb : \left|\frac{\sum_{A \in G_K} \ell\left(\CC_A \cap B_R(w)\right)}{\vol\left(B_R\right) \sum_{A \in G_K} \ell\left(\CC_A\right)} - \frac{1}{\vol\left(\Gamma \backslash \Hb\right)}\right| > c\right\}\right)\]
converges to zero as $D \to \infty$ along odd fundamental discriminants, and for all $\delta > 0$ and $c \gg_{\e} D^{-1/4 + \e}$,
\[\vol\left(\left\{w \in \Gamma \backslash \Hb : \left|\frac{\sum_{A \in G_K} \vol\left(\Gamma_A \backslash \NN_A \cap B_R(w)\right)}{\vol\left(B_R\right) \sum_{A \in G_K} \vol\left(\Gamma_A \backslash \NN_A\right)} - \frac{1}{\vol\left(\Gamma \backslash \Hb\right)}\right| > c\right\}\right)\]
converges to zero as $D \to \infty$ along odd fundamental discriminants.
\end{theorem}

The fact that these geometric invariants equidistribute on almost every ball of different scales should not come as a surprise, and essentially boils down to the fact that a Heegner point has dimension $0$, a closed geodesic has dimension $1$, and a hyperbolic orbifold has dimension $2$. For Heegner points, we need roughly $R^2$ balls to cover $\Gamma \backslash \Hb$, so we require the number of Heegner points $\# G_K$ corresponding to the genus $G_K$ to be at least $R^2$ in order to expect equidistribution; this is the scale $R \asymp (-D)^{-1/4}$. For closed geodesics, on the other hand, $R$ balls will cover roughly $1/R$ of $\Gamma \backslash \Hb$, but a closed geodesic may intersect more than one ball, so we only require the total length $\sum_{A \in G_K} \ell\left(\CC_A\right)$ of closed geodesics corresponding to the genus $G_K$ to be at least $R$; this is the scale $R \asymp D^{-1/2}$. Finally, we should expect equidistribution at \emph{all} scales for hyperbolic orbifolds, since these are just (possibly uneven) coverings of $\Gamma \backslash \Hb$.

\subsection{Idea of Proof}

The chief idea behind the proof of the aforementioned small scall equidistribution theorems is to use Chebyshev's inequality to reduce the problem to bounding a variance. For example,
\[\vol\left(\left\{w \in \Gamma \backslash \Hb : \left|\frac{1}{\vol\left(B_R\right)} \int_{B_R(w)} |g(z)|^2 \, d\mu(z) - \frac{1}{\vol\left(\Gamma \backslash \Hb\right)}\right| > c\right\}\right) \leq \frac{1}{c^2} \Var(g;R)\]
with
\[\Var(g;R) \defeq \int_{\Gamma \backslash \Hb} \left(\frac{1}{\vol\left(B_R\right)} \int_{B_R(w)} |g(z)|^2 \, d\mu(z) - \frac{1}{\vol\left(\Gamma \backslash \Hb\right)}\right)^2 \, d\mu(w).\]

The method of bounding the variance in order to show equidistribution in almost every shrinking ball is also used in \cite[Theorem 1.6]{GW} for eigenfunctions of the Laplacian on $\T^2$, as well as in both \cite[Theorem 1.3]{EMV} and \cite[Theorem 1.8]{BRS}, where the problem investigated is not quantum unique ergodicity, but rather the equidistribution of lattice points on the sphere.

The variance is an inner product of functions in $L^2(\Gamma \backslash \Hb)$, as is the fourth moment of a truncated Eisenstein series; both are thereby amenable to being spectrally expanded via Parseval's identity. The resulting spectral sum over Hecke--Maa\ss{} forms $f$ occurring in the spectral expansion $\Var(g;R)$ when $g$ is an Eisenstein series is essentially the same as the spectral sum for fourth moment of a truncated Eisenstein series in the range $0 < t_f \ll_{\e} R^{-1 + \e}$, whereas for $t_f \gg 1/R$, it is much smaller.

Finally, we use the Watson--Ichino formula to write $|\langle |g|^2,f\rangle|^2$ as a product of $L$-functions. This reduces the problem to bounding certain moments of $L$-functions, with the length of these moments corresponding inversely to the radius of the shrinking ball.

Though not a manifestation of the random wave conjecture, the equidistribution problems in \hyperref[GeomInvintrosect]{Section \ref*{GeomInvintrosect}} nonetheless involve equidistribution on $\Gamma \backslash \Hb$, and the proofs of \hyperref[quadalmosteverycondthm]{Theorems \ref*{quadalmosteverycondthm}} and \ref{quadalmosteveryuncondthm} contain many of the same ingredients as the proofs of \hyperref[Maassalmosteverythm]{Theorems \ref*{Maassalmosteverythm}} and \ref{Eisenalmosteverythm}. The chief difference is that in place of $|\langle |g|^2,f\rangle|^2$, we have Weyl sums; akin to the Watson--Ichino formula, these can be expressed as a product of $L$-functions via the work of Duke, Imamo\={g}lu, and T\'{o}th \cite{DIT}.

\subsection{Connections to Subconvexity}

The rate of equidistribution for quantum unique ergodicity for Hecke--Maa\ss{} eigenforms $g \in \BB_0(\Gamma)$ can be quantified via explicit rates of decay for
\[\int_{\Gamma \backslash \Hb} f(z) |g(z)|^2 \, d\mu(z), \qquad \int_{\Gamma \backslash \Hb} E(z,\psi) |g(z)|^2 \, d\mu(z)\]
for fixed $f \in \BB_0(\Gamma)$ and $\psi \in C_c^{\infty}(\R^{+})$ as $t_g$ tends to infinity. Via the Watson--Ichino formula, this is equivalent to obtaining subconvex bounds of the form
\[L\left(\frac{1}{2}, \sym^2 g \otimes f\right) \ll_f t_g^{1 - \delta}, \qquad L\left(\frac{1}{2} + it, \sym^2 g\right) \ll_t t_g^{\frac{1}{2}\left(1 - \delta\right)}\]
for some absolute constant $\delta > 0$. Similarly, quantifying the rate of equidistribution for quantum unique ergodicity for $g(z) = E(z,1/2 + it_g)$ is equivalent to obtaining subconvex bounds of the form
\[L\left(\frac{1}{2} + 2it_g, f\right) \ll_f t_g^{\frac{1}{2}\left(1 - \delta\right)}, \qquad \zeta\left(\frac{1}{2} + i(2t_g \pm t)\right) \ll_t t_g^{\frac{1}{4}\left(1 - \delta\right)}\]
for some absolute constant $\delta > 0$.

For quantum unique ergodicity in almost every shrinking ball of radius $R$ for Hecke--Maa\ss{} eigenforms $g \in \BB_0(\Gamma)$, on the other hand, we will show that we require bounds of the form
\[\sum_{H \leq t_f \leq 2H} \frac{L\left(\frac{1}{2}, f\right) L\left(\frac{1}{2}, \sym^2 g \otimes f\right)}{L(1,\sym^2 f)} \ll_{\delta} Ht_g^{1 - \delta}\]
for some absolute constant $\delta > 0$ \emph{uniformly in $1 \ll H \ll 1/R$}. That is, we require subconvex moment bounds for $L$-functions uniformly in two parameters: $t_f$ and $t_g$. Thus this is a problem of hybrid subconvexity. Proving such bounds unconditionally seems to be currently out of reach for moments involving $\GL_3 \times \GL_2$ Rankin--Selberg $L$-functions. For $g(z) = E(z,1/2 + it_g)$, on the other hand, the required subconvex moment bounds are
\[\sum_{H \leq t_f \leq 2H} \frac{L\left(\frac{1}{2}, f\right)^2 \left|L\left(\frac{1}{2} + 2it_g\right)\right|^2}{L(1,\sym^2 f)} \ll_{\delta} Ht_g^{1 - \delta},\]
and the fact that these moments only involve $\GL_2$ $L$-functions makes this problem tractable. It is for this reason that we are able to prove \hyperref[Eisenalmosteverythm]{Theorem \ref*{Eisenalmosteverythm}} unconditionally, whereas \hyperref[Maassalmosteverythm]{Theorem \ref*{Maassalmosteverythm}} is conditional.

\section{Integrals of Automorphic Forms and \texorpdfstring{$L$}{L}-Functions}
\label{section2}

\subsection{The Maa\ss{}--Selberg Relation}

The Eisenstein series $E(z,1/2 + it)$ is not square-integrable for any $t \in \R$. However, this is no longer the case when we replace the Eisenstein series with the truncated Eisenstein series
\[g(z) = \Lambda^T E\left(z, \frac{1}{2} + it_g\right),\]
since $\Lambda^T E(z,s)$ is of rapid decay at the cusp of $\Gamma \backslash \Hb$. Note that
\[\Lambda^T E(z,s) = \begin{dcases*}
E(z,s) & if $1/T \leq \Im(z) \leq T$,	\\
E(z,s) - \Im(z)^s + \varphi(s) \Im(z)^{1 - s} & if $\Im(z) > T$,
\end{dcases*}\]
where
\[\varphi(s) = \frac{\Lambda(2 - 2s)}{\Lambda(2s)}.\]
The following explicit formula for the inner product of two truncated Eisenstein series is known as the Maa\ss{}--Selberg relation.

\begin{proposition}[{\cite[Proposition 6.8]{Iwa}}]
For $T \geq 1$, and $s \neq \overline{r}$, $s + \overline{r} \neq 1$,
\begin{multline}\label{MaassSelberg}
\int_{\Gamma \backslash \Hb} \Lambda^T E(z,s) \overline{\Lambda^T E(z,r)} \, d\mu(z)	\\
= \frac{T^{s + \overline{r} - 1}}{s + \overline{r} - 1} + \overline{\varphi(r)} \frac{T^{s - \overline{r}}}{s - \overline{r}} + \varphi(s) \frac{T^{\overline{r} - s}}{\overline{r} - s} + \varphi(s) \overline{\varphi(r)} \frac{T^{1 - s - \overline{r}}}{1 - s - \overline{r}}.
\end{multline}
\end{proposition}

\begin{corollary}\label{constantcontrib}
We have that
\[\int_{\Gamma \backslash \Hb} \left|\Lambda^T E\left(z, \frac{1}{2} + it_g\right)\right|^2 \, d\mu(z) = \log \left(\left(\frac{1}{4} + t_g^2\right) T^2\right) + O\left( \left(\log t_g\right)^{2/3} \left(\log \log t_g\right)^{1/3}\right).\]
\end{corollary}

\begin{proof}
We take $s = r = 1/2 + it_g + \e$ with $\e > 0$ in the Maa\ss{}--Selberg relation \eqref{MaassSelberg} to obtain
\[\int_{\Gamma \backslash \Hb} \left|\Lambda^T E\left(z,\frac{1}{2} + it_g + \e\right)\right|^2 \, d\mu(z) = \frac{T^{2\e}}{2\e} - \left|\varphi\left(\frac{1}{2} + it_g + \e\right)\right|^2 \frac{T^{-2\e}}{2\e}.\]
Using the Taylor expansions
\begin{align*}
T^{2\e} & = 1 + 2\e \log T + O\left(\e^2\right),	\\
\varphi\left(\frac{1}{2} + it_g + \e\right) & = \varphi\left(\frac{1}{2} + it_g\right) + \e \varphi'\left(\frac{1}{2} + it_g\right) + O\left(\e^2\right),
\end{align*}
together with the fact that $|\varphi(1/2 + it_g)| = 1$ and that
\begin{equation}\label{varphi'varphi}
\begin{split}
\frac{\varphi'}{\varphi}\left(\frac{1}{2} + it_g\right) & = - 4\Re\left(\frac{\Lambda'}{\Lambda}\left(1 + 2it_g\right)\right)	\\
& = 2 \log \pi - 2\Re\left(\frac{\Gamma'}{\Gamma}\left(\frac{1}{2} + it_g\right)\right) - 4\Re\left(\frac{\zeta'}{\zeta}\left(1 + 2it_g\right)\right),
\end{split}
\end{equation}
we find that
\[\int_{\Gamma \backslash \Hb} |g(z)|^2 \, d\mu(z) = 2 \log T - 2 \log \pi + 2\Re\left(\frac{\Gamma'}{\Gamma}\left(\frac{1}{2} + it_g\right)\right) + 4\Re\left(\frac{\zeta'}{\zeta}\left(1 + 2it_g\right)\right).\]
It remains to use Stirling's formula to find that
\begin{equation}\label{Gamma'Gamma}
2\Re\left(\frac{\Gamma'}{\Gamma}\left(\frac{1}{2} + it_g\right)\right) = \log\left(\frac{1}{4} + t_g^2\right) + O\left(\frac{1}{t_g}\right),
\end{equation}
and \cite[Theorem 8.29]{IK} to give the bound
\begin{equation}\label{zeta'zeta}
\frac{\zeta'}{\zeta}\left(1 + 2it_g\right) \ll \left(\log t_g\right)^{2/3} \left(\log \log t_g\right)^{1/3}. \qedhere 
\end{equation}
\end{proof}

\subsection{The Watson--Ichino Formula}
\label{WatsonIchinosect}

To deal with spectral sums involving terms of the form $| \langle |g|^2, f \rangle |^2$, one can use the Watson--Ichino formula, which essentially states that the square of the integral over $\Gamma \backslash \Hb$ of the product of three automorphic forms is equal to a product of completed $L$-functions involving those automorphic forms. In particular, if $f,g \in \BB_0(\Gamma)$, then from \cite[Theorem 1.1]{Ich} and \cite[Theorem 3]{Wat},
\[\left|\left\langle |g|^2, f \right\rangle \right|^2 = \frac{\Lambda\left(\frac{1}{2}, g \otimes \widetilde{g} \otimes f\right)}{\Lambda(1, \sym^2 g)^2 \Lambda\left(1, \sym^2 f\right)}.\]
Here $\Lambda(s,\pi)$ denotes the completed $L$-function of an automorphic representation $\pi$ of $\GL_n(\A_{\Q})$: this is of the form
\begin{equation}\label{Lambda(s,pi)}
\Lambda(s,\pi) = q_{\pi}^{s/2} L_{\infty}(s,\pi) L(s,\pi),
\end{equation}
where $q_{\pi}$ denotes the conductor of $\pi$, $L_{\infty}(s,\pi)$ is the archimedean part of $\Lambda(s,\pi)$, which is of the form $\pi^{-ns/2} \prod_{j = 1}^{n} \Gamma(\frac{s + \kappa_{\pi,j}}{2})$ for some $\kappa_{\pi,j} \in \C$, and $L(s,\pi)$ is the usual nonarchimedean part of $\Lambda(s,\pi)$. Note that the numerator in the Watson--Ichino formula factorises:
\[\Lambda\left(s, g \otimes \widetilde{g} \otimes f\right) = \Lambda(s,f) \Lambda\left(s, \sym^2 g \otimes f\right).\]
Similar results also hold when either $f$ or $g$ is replaced with an Eisenstein series.

\begin{proposition}[{\cite[Equations (2.2) and (4.2)]{BK17b}}]\label{innerproductcuspasympprop}
For $f,g \in \BB_0(\Gamma)$,
\begin{align*}
\left|\left\langle |g|^2, f \right\rangle \right|^2 & = \frac{1}{8} \frac{\Lambda\left(\frac{1}{2}, f\right) \Lambda\left(\frac{1}{2}, \sym^2 g \otimes f\right)}{\Lambda\left(1, \sym^2 g\right)^2 \Lambda\left(1, \sym^2 f\right)},	\\
\left|\left\langle |g|^2, E\left(\cdot, \frac{1}{2} + it\right) \right\rangle \right|^2 & = \frac{1}{4} \frac{\Lambda\left(\frac{1}{2} + it\right) \Lambda\left(\frac{1}{2} - it\right) \Lambda\left(\frac{1}{2} + it, \sym^2 g\right) \Lambda\left(\frac{1}{2} - it, \sym^2 g\right)}{\Lambda\left(1, \sym^2 g\right)^2 \Lambda(1 + 2it) \Lambda(1 - 2it)}.
\end{align*}
\end{proposition}

A similar result also holds when $g$ is an Eisenstein series.

\begin{proposition}[{\cite[Equation (17)]{LS}, \cite[Theorem 4.1]{Spi}}]\label{innerproductEisenasympprop}
For $f \in \BB_0(\Gamma)$,
\[\left|\left\langle \left|E\left(\cdot, \frac{1}{2} + it\right)\right|^2, f \right\rangle \right|^2 = \frac{1}{2} \frac{\Lambda\left(\frac{1}{2}, f\right)^2 \Lambda\left(\frac{1}{2} + 2it_g, f\right) \Lambda\left(\frac{1}{2} - 2it_g, f\right)}{\Lambda(1 + 2it_g)^2 \Lambda\left(1 - 2it_g\right)^2 \Lambda\left(1, \sym^2 f\right)}.\]
\end{proposition}

Finally, when $f$ is also an Eisenstein series, the integral is no longer convergent. One can work around this issue by replacing this integral with a regularised integral. This is defined by Zagier \cite{Zag} in the following way. Let $F : \Gamma \backslash \Hb \to \C$ be a continuous function of moderate growth, so that there exists $c_j, \alpha_j \in \C$ and nonnegative integers $n_j$ such that
\[F(z) = \sum_{j = 1}^{\ell} \frac{c_j}{n_j!} y^{\alpha_j} (\log y)^{n_j} + O_N\left(y^{-N}\right)\]
for all $N \geq 0$ at the cusp at infinity, with no $\alpha_j$ equal to $0$ or $1$. Then there exists a function $\EE(z)$ that is a linear combination of Eisenstein series and derivatives of Eisenstein series $E(z,\alpha)$, each satisfying $\Re(\alpha) > 1/2$, such that for some $\delta > 0$,
\[F(z) - \EE(z) = O\left(y^{\frac{1}{2} - \delta}\right)\]
at the cusp at infinity. The regularised inner product of two functions $f,g$ such that $f \overline{g} = F$ is continuous and of moderate growth is defined to be
\[\langle f, g \rangle_{\reg} \defeq \int_{\Gamma \backslash \Hb} \left(F(z) - \EE(z)\right) \, d\mu(z).\]
Moreover, if $f$ and $g$ depend on complex parameters, then we may extend both sides via analytic continuation where possible.

\begin{proposition}[{\cite[Equation (44)]{Zag}}]\label{innerproductEisen3asympprop}
We have that
\begin{multline}\label{regEEE}
\left\langle E(\cdot,s_1) E(\cdot,s_2), E\left(\cdot, s\right) \right\rangle_{\reg}	\\
= \frac{\Lambda\left(\overline{s} + s_1 + s_2 - 1\right) \Lambda\left(\overline{s} + s_1 - s_2\right) \Lambda\left(\overline{s} - s_1 + s_2\right) \Lambda\left(\overline{s} - s_1 - s_2 + 1\right)}{\Lambda\left(2\overline{s}\right) \Lambda\left(2s_1\right) \Lambda\left(2s_2\right)}.
\end{multline}
\end{proposition}

In practice, it is the nonarchimedean part $L(s,\pi)$ of a completed $L$-function $\Lambda(s,\pi)$ that is difficult to deal with; this is because the asymptotic behaviour of the archimedean part of a completed $L$-function can be inferred via Stirling's approximation.

\begin{lemma}\label{Gammafactorslemma}
The product of the archimedean parts of the completed $L$-functions in \hyperref[innerproductcuspasympprop]{Propositions \ref*{innerproductcuspasympprop}}, \ref{innerproductEisenasympprop} (with $t = t_f$), and \ref{innerproductEisen3asympprop} (with $s_1 = s_2 = 1/2 + it_g$ and $s = 1/2 + it_f$) is equal to
%$\pi$ times
%\[\frac{\Gamma\left(\frac{\frac{1}{2} + it_f}{2}\right)^2 \Gamma\left(\frac{\frac{1}{2} - it_f}{2}\right)^2 \Gamma\left(\frac{\frac{1}{2} + i(2t_g + t_g)}{2}\right) \Gamma\left(\frac{\frac{1}{2} + i(2t_g - t_f)}{2}\right) \Gamma\left(\frac{\frac{1}{2} - i(2t_g + t_f)}{2}\right) \Gamma\left(\frac{\frac{1}{2} - i(2t_g - 2t_f)}{2}\right)}{\Gamma\left(\frac{1 + 2it_g}{2}\right)^2 \Gamma\left(\frac{1 - 2it_g}{2}\right)^2 \Gamma\left(\frac{1 + 2it_f}{2}\right) \Gamma\left(\frac{1 - 2it_f}{2}\right)}.\]
%The denominator simplifies to
%\[\frac{\pi^2}{\cosh^2 \pi t_g} \frac{\pi}{\cosh \pi t_f} \sim \frac{8 \pi^3}{e^{\pi(2t_g + t_f)}}.\]
%For the numerator, we use Stirling's approximation, which states that
%\[|\Gamma(\sigma + i\gamma)|^2 = 2\pi \left(\sigma^2 + \gamma^2\right)^{\sigma - 1/2} e^{-\pi |\gamma|} \left(1 + O\left(\frac{1}{1 + |\gamma|}\right)\right).\]
%It follows that when $t_f \in [0,\infty)$, the numerator is asymptotic to
%\[16 \pi^4 \left(\frac{1}{16} + \frac{t_f^2}{4}\right)^{-1/2} e^{-\pi t_f} \left(\frac{1}{16} + \frac{(2t_g + t_f)^2}{4}\right)^{-1/4} e^{-\frac{\pi}{2} (2t_g + t_f)} \left(\frac{1}{16} + \frac{(2t_g - t_f)^2}{4}\right)^{-1/4} e^{-\frac{\pi}{2} |2t_g - t_f|}.\]
\begin{multline}\label{Gammafactors}
\frac{8\pi^2 e^{-\pi \Omega(t_f,t_g)}}{(1 + t_f) (1 + 2t_g + t_f)^{1/2} (1 + |2t_g - t_f|)^{1/2}}	\\
\times \left(1 + O\left(\frac{1}{1 + t_f} + \frac{1}{1 + 2t_g + t_f} + \frac{1}{1 + |2t_g - t_f|}\right)\right),
\end{multline}
where
\[\Omega(t_f,t_g) \defeq \begin{dcases*}
0 & if $0 < t_f \leq 2t_g$,	\\
t_f - 2t_g & if $t_f > 2t_g$.
\end{dcases*}\]
\end{lemma}

\begin{proof}
The product of the archimedean parts of the completed $L$-functions is
\begin{multline*}
\pi \frac{\Gamma\left(\frac{1}{4} + \frac{i(2t_g + t_f)}{2}\right) \Gamma\left(\frac{1}{4} + \frac{i(2t_g - t_f)}{2}\right) \Gamma\left(\frac{1}{4} - \frac{i(2t_g + t_f)}{2}\right) \Gamma\left(\frac{1}{4} - \frac{i(2t_g - t_f)}{2}\right)}{\Gamma\left(\frac{1}{2} + it_g\right)^2 \Gamma\left(\frac{1}{2} - it_g\right)^2}	\\
\times \frac{\Gamma\left(\frac{1}{4} + \frac{it_f}{2}\right)^2 \Gamma\left(\frac{1}{4} - \frac{it_f}{2}\right)^2}{\Gamma\left(\frac{1}{2} + it_f\right) \Gamma\left(\frac{1}{2} - it_f\right)}.
\end{multline*}
The result then follows directly from Stirling's approximation.
\end{proof}

On occasion, we also need to deal with lower bounds for $L(1,\sym^2 f)$. This is less complex than values of $L$-functions within the critical strip $0 < \Re(s) < 1$; indeed, the following is known.

\begin{lemma}[Hoffstein--Lockhart {\cite{HL}}]
For $f \in \BB_0(\Gamma)$,
\[L\left(1,\sym^2 f\right) \gg \frac{1}{\log(t_f + 3)}.\]
\end{lemma}

\section{Sharp Bounds for the \texorpdfstring{$L^4$}{L\9040\164}-Norm of a Truncated Eisenstein Series}

\subsection{The Spectral Expansion of the \texorpdfstring{$L^4$}{L\9040\164}-Norm}

We wish to determine sharp bounds for
\[\|g\|_{L^4\left(\Gamma \backslash \Hb\right)}^4 = \int_{\Gamma \backslash \Hb} |g(z)|^4 \, d\mu(z)\]
with $g(z) = \Lambda^T E(z, 1/2 + it_g)$ in terms of $t_g$. Our first step is to express this quantity as a spectral sum, which requires the spectral decomposition of $L^2\left(\Gamma \backslash \Hb\right)$.

\begin{lemma}[{\cite[Theorem 15.5]{IK}}]\label{spectralexpansionlemma}
Let
\[f_0(z) \defeq \frac{1}{\sqrt{\vol\left(\Gamma \backslash \Hb\right)}},\]
so that $\langle f_0, f_0 \rangle = 1$, and let $\BB_0(\Gamma)$ be an orthonormal basis of Maa\ss{} cusp forms in $L^2\left(\Gamma \backslash \Hb\right)$. Then a function $g \in L^2\left(\Gamma \backslash \Hb\right)$ has the spectral expansion, valid in the $L^2$-sense, of the form
\begin{multline*}
g(z) = \left\langle g, f_0 \right\rangle f_0(z) + \sum_{f \in \BB_0(\Gamma))} \left\langle g,f \right\rangle f(z)	\\
+ \frac{1}{4\pi} \int_{-\infty}^{\infty} \left\langle g, E\left(\cdot, \frac{1}{2} + it\right) \right\rangle E\left(z, \frac{1}{2} + it\right) \, dt.
\end{multline*}
Moreover, Parseval's identity holds:
\begin{multline*}
\left\langle g_1, g_2 \right\rangle = \left\langle g_1, f_0 \right\rangle \left\langle f_0, g_2 \right\rangle + \sum_{f \in \BB_0(\Gamma)} \left\langle g_1, f \right\rangle \left\langle f, g_2 \right\rangle	\\
+ \frac{1}{4\pi} \int_{-\infty}^{\infty} \left\langle g_1, E\left(\cdot, \frac{1}{2} + it\right) \right\rangle \left\langle E\left(\cdot, \frac{1}{2} + it\right), g_2 \right\rangle \, dt
\end{multline*}
for $g_1, g_2 \in L^2\left(\Gamma \backslash \Hb\right)$.
\end{lemma}

In particular, the following spectral expansion of the $L^4$-norm of $g$ is simply Parseval's identity with $g_1 = g_2 = |g|^2$.

\begin{corollary}\label{L4spectralcor}
Let $g \in L^2\left(\Gamma \backslash \Hb\right)$ be of rapid decay. Then
\[\|g\|_{L^4\left(\Gamma \backslash \Hb\right)}^4 = \left| \left\langle |g|^2,f_0 \right\rangle \right|^2 + \sum_{f \in \BB_0(\Gamma)} \left| \left\langle |g|^2,f \right\rangle \right|^2 + \frac{1}{4\pi} \int_{-\infty}^{\infty} \left| \left\langle |g|^2, E\left(\cdot, \frac{1}{2} + it\right) \right\rangle \right|^2 \, dt.\]
\end{corollary}

This is reduced to understanding bounds for the inner product of $|g|^2$ with eigenfunctions of the Laplacian. The first term in this expansion is the inner product of $|g|^2$ with the constant function
\[f_0(z) = \frac{1}{\sqrt{\vol\left(\Gamma \backslash \Hb\right)}},\]
and \hyperref[constantcontrib]{Corollary \ref*{constantcontrib}} shows that
\[\left| \left\langle |g|^2, f_0 \right\rangle \right|^2 = \frac{\left(\log \left(\frac{1}{4} + t_g^2\right)\right)^2}{\vol\left(\Gamma \backslash \Hb\right)} + O_T\left(\left(\log t_g\right)^{5/3} \left(\log \log t_g\right)^{1/3}\right).\]
It remains to treat the cuspidal and continuous spectra.

\subsection{Ranges of the Spectral Decomposition for the \texorpdfstring{$L^4$}{L\9040\164}-Norm}\label{spectralranglesL4sect}

We divide the spectral expansion of the $L^4$-norm of $g(z) = \Lambda^T E(z,1/2 + it_g)$ given in \hyperref[L4spectralcor]{Corollary \ref*{L4spectralcor}} into different parts, then analyse each part individually.

There are two main ranges of the continuous spectrum to consider, which depend on a small fixed parameter $\delta > 0$:
\begin{itemize}
\item the initial range $0 \leq |t| \leq 2t_g + t_g^{1 - \delta}$, and
\item the tail range $|t| > 2t_g + t_g^{1 - \delta}$.
\end{itemize}
Both of these ranges will be shown to contribute a negligible amount via subconvexity estimates for the $L$-functions appearing in the integral.

For the contribution from the cuspidal spectrum, the summation over $\BB_0(\Gamma)$ may be broken up into different ranges depending on $t_f$. There are four main ranges of the cuspidal spectrum left to consider, which depend on a fixed small parameter $\delta > 0$:
\begin{itemize}
\item the short initial range $0 \leq t_f \leq t_g^{1 - \delta}$,
\item the bulk range $t_g^{1 - \delta} < t_f < 2t_g - t_g^{1 - \delta}$,
\item the short transition range $2t_g - t_g^{1 - \delta} \leq t_f \leq 2t_g + t_g^{1 - \delta}$, and
\item the tail range $t_f > 2t_g + t_g^{1 - \delta}$.
\end{itemize}
We divide the spectral sum into these particular ranges due to the size of the product of analytic conductors of $L$-functions. The analytic conductor of
\[L\left(\frac{1}{2}, f\right)^2 L\left(\frac{1}{2} + 2it_g, f\right) L\left(\frac{1}{2} - 2it_g, f\right)\]
is approximately
\[\left(\frac{1}{4} + t_f^2\right)^2 \left(\frac{1}{4} + \left(2t_g^2 + t_f^2\right)\right) \left(\frac{1}{4} + \left|2t_g^2 - t_f^2\right|\right),\]
which is large when $t_f$ lies in the bulk range, but is small in the short initial range, and drops in the short transition range. For this reason, the main contribution will be shown to arise from the bulk range, while the contribution from the two short ranges will be shown to be negligible. Assuming the generalised Lindel\"{o}f hypothesis, this can be proven directly; see \cite[Section 5]{BK17b}. Finally, the exponential decay in \eqref{Gammafactors} arising from the archimedean components of the completed $L$-functions indicates that the tail range contributes a negligible amount.

\begin{remark}\label{SpinuSarnakremark}
In \cite[Chapter 6]{Spi}, Spinu sketches an unconditional proof of \hyperref[L4uncondthm]{Theorem \ref*{L4uncondthm}}. The proof, however, only treats the spectral sum in the range $\alpha t_g < t_f < 2(1 - \alpha) t_g$ for any fixed $\alpha > 0$ (essentially the bulk range), in which the contribution of the spectral sum ought to be nonnegligible. The remaining ranges, which all ought to contribute a negligible amount, are left unaddressed.

This same issue is present in a claim of Sarnak and Watson \cite[Theorem 3(a)]{Sar} of the bound $\|g\|_{L^4(\Gamma \backslash \Hb)} \ll_{\e} t_g^{\e}$ for Hecke--Maa\ss{} cusp forms, under the assumption of the Selberg eigenvalue and Ramanujan conjectures (but not the generalised Lindel\"{o}f hypothesis, as in \cite[Theorem 1.1]{BK17b}). Sarnak (personal communication) subsequently has retracted this claim, and instead only claims this bound for the contribution of the spectral sum in the bulk range, as the method he uses is unable to treat the short initial range.

We are able to treat the short initial and transition ranges, left untreated by Spinu, by applying the work of Jutila \cite{Jut04}, Ivi\'{c} \cite{Ivi}, and Jutila and Motohashi \cite{JM} on certain hybrid moments of $L$-functions. We do not know how to treat these ranges when $g$ is a Hecke--Maa\ss{} cusp form.
\end{remark}

\subsection{Spectral Methods to Bound the Continuous Spectrum}

From \hyperref[L4spectralcor]{Corollary \ref*{L4spectralcor}}, we must bound
\begin{equation}\label{continuousspectrumeq}
\frac{1}{4\pi} \int_{-\infty}^{\infty} \left| \left\langle |g|^2, E\left(\cdot, \frac{1}{2} + it\right) \right\rangle \right|^2 \, dt.
\end{equation}

\begin{lemma}[{\cite[Theorem 3.3]{Spi}}]
There exists a positive constant $c > 0$ such that \eqref{continuousspectrumeq} is bounded by $108 T + O\left(t_g^{-c}\right)$.
\end{lemma}

Here $c$ is any constant less than $1/2 - 2\theta$, where $\theta$ is a positive constant such that
\[\zeta\left(\frac{1}{2} + it\right) \ll_{\e} (|t| + 1)^{\theta + \e}.\]
The best bound known is $\theta = 13/84$, due to Bourgain \cite[Theorem 5]{Bou}.

\subsection{Reduction to Untruncated Eisenstein Series for the Cuspidal Spectrum}

From \hyperref[L4spectralcor]{Corollary \ref*{L4spectralcor}}, we must bound
\[\sum_{f \in \BB_0(\Gamma)} \left| \left\langle |g|^2, f \right\rangle \right|^2.\]
First, we observe that $g(z) = \Lambda^T E\left(z, 1/2 + it_g\right)$ can be replaced by $E\left(z, 1/2 + it_g\right)$.

\begin{lemma}[{\cite[Theorem 4.2]{Spi}}]
We have that
\[\sum_{f \in \BB_0(\Gamma)} \left| \left\langle |g|^2, f \right\rangle \right|^2 \leq \sum_{f \in \BB_0(\Gamma)} \left| \left\langle \left|E\left(\cdot, \frac{1}{2} + it_g\right)\right|^2, f \right\rangle \right|^2 + O_T\left(\left(\log t_g\right)^2\right).\]
\end{lemma}

This allows us to use \hyperref[innerproductEisenasympprop]{Proposition \ref*{innerproductEisenasympprop}} and \hyperref[Gammafactorslemma]{Lemma \ref*{Gammafactorslemma}}. We divide the cuspidal spectrum into four ranges, as discussed in \hyperref[spectralranglesL4sect]{Section \ref*{spectralranglesL4sect}}. The convexity bound for the associated $L$-functions together with the Weyl law shows that the tail range is negligible. So it remains to bound the first three ranges.

\subsection{Weaker Bounds via the Large Sieve}

In \cite[Chapter 5]{Spi}, Spinu proves the following moment bounds in dyadic intervals, a corollary of which is the bound $\|g\|_{L^4(\Gamma \backslash \Hb)} \ll_{\e} t_g^{\e}$.

\begin{lemma}[{\cite[Proposition 5.4]{Spi}}]\label{largesieveinitiallemma}
We have that
\[\sum_{H \leq t_f \leq 2H} L\left(\frac{1}{2}, f\right)^2 \left|L\left(\frac{1}{2} + 2it_g, f\right)\right|^2 \ll_{\e} H t_g^{1 + \e}\]
uniformly in $H \leq 2t_g - t_g^{1 - \delta}$.
\end{lemma}

\begin{lemma}[{\cite[Proposition 5.5]{Spi}}]\label{largesievetransitionlemma}
We have that
\[\sum_{H < |t_f - 2t_g| < 2H} L\left(\frac{1}{2}, f\right)^2 \left|L\left(\frac{1}{2} + 2it_g, f\right)\right|^2 \ll H^{1/2} t_g^{\frac{3 + \delta}{2}}\]
uniformly in $1 \leq H \ll t_g^{1 - \delta}$.
\end{lemma}

\begin{remark}
Spinu uses the large sieve only to prove \hyperref[largesieveinitiallemma]{Lemma \ref*{largesieveinitiallemma}} and employs a more complex method in proving \hyperref[largesievetransitionlemma]{Lemma \ref*{largesievetransitionlemma}}; nonetheless, one can in fact use the local large sieve, as stated in \cite[Lemma]{Luo}, to prove the latter; see \cite[Proof of Theorem]{Luo}.
\end{remark}

\subsection{Spectral Methods to Bound the Short Initial Range}

From \cite[Theorem 8.29]{IK}, we have that bound
\[\frac{1}{\zeta(1 + it)} \ll (\log t)^{2/3} (\log \log t)^{1/3}.\]
It therefore suffices to show that
\[\sum_{0 < t_f < t_g^{1 - \delta}} \frac{L\left(\frac{1}{2}, f\right)^2 \left| L\left(\frac{1}{2} + 2it_g, f\right)\right|^2}{\left(1 + t_f\right) \left(1 + 2t_g + t_f\right)^{1/2} \left(1 + 2t_g - t_f\right)^{1/2} L\left(1, \sym^2 f\right)} \ll t_g^{-\delta'}\]
for some $\delta' > 0$. We divide the short transition range $0 < t_f < t_g^{1 - \delta}$ into dyadic intervals $H \leq t_f < 2H$, of which there are roughly $\log t_g$ intervals, on which
\[(1 + t_f) (1 + 2t_g + t_f)^{1/2} (1 + 2t_g - t_f)^{1/2} \asymp H t_g.\]
It then suffices to show that for $H \ll t_g^{1 - \delta}$,
\[\sum_{H \leq t_f \leq 2H} \frac{L\left(\frac{1}{2}, f\right)^2 \left|L\left(\frac{1}{2} + 2it_g, f\right)\right|^2}{L\left(1, \sym^2 f\right)} \ll H t_g^{1 - \delta'}.\]
This bound follows from the work of Jutila \cite{Jut04}, Ivi\'{c} \cite{Ivi}, and Jutila and Motohashi \cite{JM}. It is worth noting that the purpose of these works is to obtain Weyl-type subconvexity bounds
\[L\left(\frac{1}{2} + it, f\right) \ll_{\e} \qq\left(f,\frac{1}{2} + it\right)^{\frac{1}{6} + \e}\]
for Hecke--Maa\ss{} eigenforms $f \in \BB_0(\Gamma)$, so long as $|t|$ is not too close to $t_f$; here $\qq(f,s)$ denotes the analytic conductor of $L(s,f)$. Conveniently, their methods to obtain such bounds involve obtaining bounds for the exact type of spectral sum that we are studying.

\begin{lemma}\label{shortinitialrangelemma}
For $t \geq 0$ and $H \gg 1$, we have that
\[\sum_{H \leq t_f \leq 2H} \frac{L\left(\frac{1}{2}, f\right)^2 \left|L\left(\frac{1}{2} + it, f\right)\right|^2}{L\left(1, \sym^2 f\right)} \ll_{\e} \begin{dcases*}
H^{2 + \e} & if $H \geq t^{2/3}$,	\\
t^{\frac{4}{3} + \e} & if $t^{1/2} \leq H \leq t^{2/3}$,	\\%This gives $t_g^{-1/6 + \e}$
H^{\frac{8}{3} + \e} & if $t^{1/3} \leq H \leq t^{1/2}$,	\\%This gives $t_g^{-1/6 + \e}$ 
H^{\frac{2}{3} + \e} t^{\frac{2}{3} + \e} & if $H \leq t^{1/3}$.%This gives $t_g^{-1/3 + \e}$.
\end{dcases*}\]
\end{lemma}
%
%In particular, the desired result holds for any $\delta < 1/6$.

\begin{proof}
For $H \geq t^{1/2}$, this follows from \cite[Theorem 2]{JM}, which states that for $t \geq 0$ and $H \gg 1$,
\[\sum_{H \leq t_f \leq 2H} \frac{L\left(\frac{1}{2}, f\right)^2 \left|L\left(\frac{1}{2} + it, f\right)\right|^2}{L\left(1, \sym^2 f\right)} \ll_{\e} \left(H^2 + t^{4/3}\right)^{1 + \e}.\]
For $H \leq t^{1/2}$, this follows from the subconvexity bound
\[L\left(\frac{1}{2}, f\right) \ll_{\e} t_f^{\frac{1}{3} + \e}\]
of Ivi\'{c} \cite[Corollary 2]{Ivi}, and from \cite[Theorem]{Jut04}, which states that for $t \geq 0$ and $1 \ll G \ll H$,
\[\sum_{H \leq t_f \leq H + G} \frac{\left|L\left(\frac{1}{2} + it, f\right)\right|^2}{L\left(1, \sym^2 f\right)} \ll_{\e} \left(GH + t^{2/3}\right)^{1 + \e}.\qedhere\]
\end{proof}

\begin{corollary}\label{shortinitialrangecor}
For any $\delta > 0$, we have that
\[\sum_{0 < t_f < t_g^{1 - \delta}} \frac{\Lambda\left(\frac{1}{2}, f\right)^2 \Lambda\left(\frac{1}{2} + 2it_g, f\right) \Lambda\left(\frac{1}{2} - 2it_g, f\right)}{\Lambda(1 + 2it_g)^2 \Lambda(1 - 2it_g)^2 \Lambda\left(1, \sym^2 f\right)} \ll_{\e} t_g^{-\min\left\{\delta, \frac{1}{6}\right\} + \e}.\]
\end{corollary}

\subsection{Spectral Methods to Bound the Short Transition Range}

In \cite[Section 1]{BK17a}, Buttcane and Khan state for a dihedral Maa\ss{} newform $g$,

\begin{quote}
\ldots the range [$2t_g - t_g^{1 - \delta} < t_f < 2t_g$] can be handled by applying H\"{o}lder's inequality as Luo does and then applying Jutila's \cite{Jut01} and Ivi\'{c}'s  \cite{Ivi} bounds for moments of $L(1/2,f)$ in short intervals of $t_f$ close to $2t_g$.
\end{quote}

A similar idea works when $g$ is a truncated Eisenstein series. We must show that
\[\sum_{2t_g - t_g^{1 - \delta} \leq t_f \leq 2t_g + t_g^{1 - \delta}} \frac{L\left(\frac{1}{2}, f\right)^2 \left|L\left(\frac{1}{2} + 2it_g, f\right)\right|^2}{\left(1 + t_f\right) \left(1 + 2t_g + t_f\right)^{1/2} \left(1 + |2t_g - t_f|\right)^{1/2} L\left(1, \sym^2 f\right)} \ll t_g^{-\delta'}\]
for some $\delta' > 0$. We use the Cauchy--Schwarz inequality to see that this spectral sum is bounded by $t_g^{-3/2}$ times the square root of the product of
\[\sum_{2t_g - t_g^{1 - \delta} \leq t_f \leq 2t_g + t_g^{1 - \delta}} \frac{L\left(\frac{1}{2}, f\right)^4}{\left(1 + |2t_g - t_f|\right)^{1/2} L\left(1, \sym^2 f\right)}\]
and
\[\sum_{2t_g - t_g^{1 - \delta} \leq t_f \leq 2t_g + t_g^{1 - \delta}} \frac{\left|L\left(\frac{1}{2} + 2it_g, f\right)\right|^4}{\left(1 + |2t_g - t_f|\right)^{1/2} L\left(1, \sym^2 f\right)}.\]

The first sum is bounded by
\begin{multline*}
\sum_{k = 0}^{\lfloor t_g^{2/3 - \delta} \rfloor} \frac{1}{\left(1 + k t_g^{1/3}\right)^{1/2}} \sum_{2 t_g - (k + 1) t_g^{1/3} \leq t_f < 2 t_g - k t_g^{1/3}} \frac{L\left(\frac{1}{2}, f\right)^4}{L\left(1, \sym^2 f\right)}	\\
+ \sum_{k = 0}^{\lfloor t_g^{2/3 - \delta} \rfloor} \frac{1}{\left(1 + k t_g^{1/3}\right)^{1/2}} \sum_{2 t_g + k t_g^{1/3} \leq t_f < 2 t_g + (k + 1) t_g^{1/3}} \frac{L\left(\frac{1}{2}, f\right)^4}{L\left(1, \sym^2 f\right)},
\end{multline*}
and a similar expression holds for the second sum. We then apply the following lemma to show that each sum is bounded by a constant multiple dependent on $\e$ of $t_g^{\frac{3 - \delta}{2} + \e}$, from which the result follows.

\begin{lemma}[{\cite[Theorem]{Jut01}, \cite[Theorem 1]{JM}}]\label{shorttransitionrangelemma}
For $H \gg 1$ and $1 \ll G \ll H$, we have that
\[\sum_{H \leq t_f \leq H + G} \frac{L\left(\frac{1}{2}, f\right)^4}{L\left(1, \sym^2 f\right)} \ll_{\e} \left(H^{1/3} + G\right) H^{1 + \e}.\]
Similarly, for $H \gg 1$, $0 \leq t \ll H^{3/2 - \e}$, and $0 \leq G \leq (H + t)^{4/3} H^{-1 + \e}$, we have that
\[\sum_{H \leq t_f \leq H + G} \frac{\left|L\left(\frac{1}{2} + it, f\right)\right|^4}{L\left(1, \sym^2 f\right)} \ll_{\e} (H + t)^{4/3} H^{\e}.\]
\end{lemma}

\begin{corollary}
For any $0 < \delta < 2/3$, we have that
\[\sum_{2t_g - t_g^{1 - \delta} \leq t_f \leq 2t_g + t_g^{1 - \delta}} \frac{\Lambda\left(\frac{1}{2}, f\right)^2 \Lambda\left(\frac{1}{2} + 2it_g, f\right) \Lambda\left(\frac{1}{2} - 2it_g, f\right)}{\Lambda(1 + 2it_g)^2 \Lambda(1 - 2it_g)^2 \Lambda\left(1, \sym^2 f\right)} \ll_{\e} t_g^{-\frac{\delta}{2} + \e}.\]
\end{corollary}

\subsection{Spectral Methods to Bound the Bulk Range}

In \cite[Chapter 6]{Spi}, Spinu proves the bound
\[\sum_{2 \alpha t_g \leq t_f \leq 2(1 - \alpha) t_g} \frac{\Lambda\left(\frac{1}{2}, f\right)^2 \Lambda\left(\frac{1}{2} + 2it_g, f\right) \Lambda\left(\frac{1}{2} - 2it_g, f\right)}{\Lambda(1 + 2it_g)^2 \Lambda(1 - 2it_g)^2 \Lambda\left(1, \sym^2 f\right)} \ll_{\alpha} \left(\log \left(\frac{1}{4} + t_g^2\right)\right)^2\]
for any small $\alpha > 0$. Via the methods of Buttcane and Khan \cite{BK17a,BK17b} (the chief difference of which is using a different test function in the Kuznetsov formula), this extends to the full bulk range $t_g^{1 - \delta} < t_f < 2t_g - t_g^{1 - \delta}$, which thereby completes the unconditional proof of \hyperref[L4uncondthm]{Theorem \ref*{L4uncondthm}}.

\section{Failure of Equidistribution at the Planck Scale}

\subsection{The Selberg--Harish-Chandra Transform}

For $z,w \in \Hb$, set
\[u(z,w) \defeq \frac{|z - w|^2}{4 \Im(z) \Im(w)} = \sinh^2 \frac{\rho(z,w)}{2},\]
where
\[\rho(z,w) \defeq \log \frac{\left|z - \overline{w}\right| + |z - w|}{\left|z - \overline{w}\right| - |z - w|}\]
denotes the hyperbolic distance on $\Hb$. The function $u : \Hb \times \Hb \to [0,\infty)$ is a point-pair invariant. From this, a function $k : [0,\infty) \to \C$ gives rise to a point-pair invariant $k(z,w) \defeq k(u(z,w))$ on $\Hb$. The Selberg--Harish-Chandra transform maps sufficiently well-behaved functions $k : [0,\infty) \to \C$ to functions $h : \R \to \C$. This transform is given in three steps as follows:
\[q(v) \defeq \int_{v}^{\infty} \frac{k(u)}{\sqrt{u - v}} \, du, \qquad g(r) \defeq 2 q\left(\sinh^2 \frac{r}{2}\right), \qquad h(t) \defeq \int_{-\infty}^{\infty} g(r) e^{irt} \, dr.\]
Note that $h(t)$ is real whenever $t$ is real.

We shall take $k(z,w) = k_R(z,w)$ equal to the indicator function of a small ball of radius $R$ centred at a point $w$,
\[B_R(w) \defeq \{z \in \Hb : \rho(z,w) \leq R\} = \left\{z \in \Hb : u(z,w) \leq \sinh^2 \frac{R}{2}\right\},\]
normalised by the volume of this ball. So
\begin{equation}\label{k(u)def}
k(u) = k_R(u) \defeq \begin{dcases*}
\dfrac{1}{4\pi \sinh^2 \frac{R}{2}} & if $u \leq \sinh^2 \dfrac{R}{2}$,	\\
0 & otherwise,
\end{dcases*}
\end{equation}
and consequently
%\begin{align*}
%q_R(v) & \defeq \int_{v}^{\infty} \frac{k_R(u)}{\sqrt{u - v}} \, du = \begin{dcases*}
%\frac{\sqrt{\sinh^2 \frac{R}{2} - v}}{2 \pi \sinh^2 \frac{R}{2}} & if $v \leq \sinh^2 \dfrac{R}{2}$,	\\
%0 & otherwise,
%\end{dcases*}	\\
%g_R(r) & \defeq 2 q_R\left(\sinh^2 \frac{r}{2}\right) = \begin{dcases*}
%\frac{\sqrt{\sinh^2 \frac{R}{2} - \sinh^2 \frac{r}{2}}}{\pi \sinh^2 \frac{R}{2}} & if $|r| \leq R$,	\\
%0 & otherwise,
%\end{dcases*}	\\
%h_R(t) & \defeq \int_{-\infty}^{\infty} g_R(r) e^{irt} \, dr = \frac{1}{\pi \sinh^2 \frac{R}{2}} \int_{-R}^{R} \sqrt{\sinh^2 \frac{R}{2} - \sinh^2 \frac{r}{2}} e^{irt} \, dr.
%\end{align*}
\[h(t) = h_R(t) \defeq \frac{R}{\pi \sinh \frac{R}{2}} \int_{-1}^{1} \sqrt{1 - \left(\frac{\sinh \frac{Rr}{2}}{\sinh \frac{R}{2}}\right)^2} e^{iRrt} \, dr.\]

We require the following asymptotics for $h_R(t)$, which are extremely similar to the analogous result for $\T^2$; see \cite[Lemma 2.1]{GW}.

\begin{lemma}[{Cf.~\cite[Lemma 2.4]{Cha96}}]\label{hR(t)asymplemma}
As $R$ tends to zero, we have that
\[h_R(t) \sim \begin{dcases*}
1 & if $Rt$ tends to zero,	\\
\frac{2 J_1(Rt)}{Rt} & if $Rt \in (0,\infty)$,	\\
\frac{1}{\sqrt{\pi}} \left(\frac{2}{Rt}\right)^{3/2} \sin\left(Rt - \frac{\pi}{4}\right) & if $Rt$ tends to infinity.
\end{dcases*}\]
\end{lemma}

\begin{proof}
If $R$ and $Rt$ both converge to zero, then the dominated convergence theorem implies that
\[h_R(t) \sim \frac{2}{\pi} \int_{-1}^{1} \sqrt{1 - r^2} \, dr = 1.\]
If $R$ converges to $0$ and $Rt$ converges to some value in $(0,\infty)$, then similarly
\[h_R(t) \sim \frac{2}{\pi} \int_{-1}^{1} \sqrt{1 - r^2} e^{iRrt} \, dr = \frac{2 J_1(Rt)}{Rt}\]
via \cite[8.411.10]{GR}. So it remains to prove the case that $R$ converges to $0$ and $Rt$ tends to infinity. To do this, we let
\[h(R,x) \defeq \frac{R}{\pi \sinh \frac{R}{2}} \int_{-1}^{1} \sqrt{1 - \left(\frac{\sinh \frac{Rr}{2}}{\sinh \frac{R}{2}}\right)^2} e^{irx} \, dr.\]
We show that
\[x^{3/2} h(R,x) - 2 \sqrt{\frac{2}{\pi} \frac{R}{\sinh R}}\sin\left(x - \frac{\pi}{4}\right)\]
is pointwise convergent as $R$ tends to zero and is uniformly convergent to $0$ as $x$ tends to infinity, from which the Moore--Osgood theorem allows us to interchange the order of limits taken in order to obtain the desired asymptotic. Indeed, the dominated convergence theorem once again shows that $h(R,x)$ converges to
\[\frac{2}{\pi} \int_{-1}^{1} \sqrt{1 - r^2} e^{irx} \, dr = \frac{2 J_1(x)}{x}\]
as $R$ tends to zero. For the uniform convergence as $x$ tends to infinity, we integrate by parts 
%in order to find that
%\[h(R,x) = \frac{R}{2 \sinh \frac{R}{2}} \frac{2}{\pi ix} \int_{-1}^{1} \frac{\sinh \frac{Rr}{2}}{\sinh \frac{R}{2}} e^{irx} \frac{R}{2 \sinh \frac{R}{2}} \frac{\cosh \frac{Rr}{2}}{\sqrt{1 - \left(\frac{\sinh \frac{Rr}{2}}{\sinh \frac{R}{2}}\right)^2}} \, dr,\]
and make the substitution $r = \frac{2}{R} \arsinh\left(\sin v \sinh \frac{R}{2}\right)$, 
%so that $\sin v = \frac{\sinh \frac{Rr}{2}}{\sinh \frac{R}{2}}$ and hence $dv = \frac{R}{2 \sinh \frac{R}{2}} \frac{\cosh \frac{Rr}{2}}{\sqrt{1 - \left(\frac{\sinh \frac{Rr}{2}}{\sinh \frac{R}{2}}\right)^2}} \, dr$, with $v(-1) = -\pi/2$ and $v(1) = \pi/2$,
yielding
\[h(R,x) = \frac{R}{2 \sinh \frac{R}{2}} \frac{2}{\pi ix} \int_{-\pi/2}^{\pi/2} \sin v e^{ix \frac{2}{R} \arsinh \left(\sin v \sinh \frac{R}{2}\right)} \, dv.\]
%We define
%\[\Psi(v) \defeq \sin v, \qquad \Phi(v,R) \defeq \frac{2}{R} \arsinh \left(\sin v \sinh \frac{R}{2}\right),\]
%so that
%\begin{align*}
%\frac{\dee \Phi}{\dee v}(v,R) & = \frac{2 \sinh \frac{R}{2}}{R} \frac{\cos v}{\sqrt{\sin^2 v \sinh^2 \frac{R}{2} + 1}},	\\
%\frac{\dee^2 \Phi}{\dee v^2}(v,R) & = -\frac{2 \sinh \frac{R}{2}}{R} \frac{\sin x \cosh^2 \frac{R}{2}}{\left(\sin^2 x \sinh^2 \frac{R}{2} + 1\right)^{3/2}}.
%\end{align*}
%The function $\Phi(v,R)$ has two critical points in the interval $[-\pi/2,\pi/2]$, namely the endpoints $v_1 = -\pi/2$ and $v_2 = \pi/2$. So by stationary phase,
%\begin{multline*}
%\int_{-\pi/2}^{\pi/2} \Psi(v) e^{ix \Phi(v,R)} \, dv \sim \Psi(v_1) e^{ix \Phi(v_1,R)} \sqrt{\frac{\pi}{2x \left|\frac{\dee^2 \Phi}{\dee v^2}(v_1,R)\right|}} e^{\frac{i\pi}{4} \sgn\left(\frac{\dee^2 \Phi}{\dee v^2}(v_1,R)\right)}	\\
%+ \Psi(v_2) e^{ix \Phi(v_2,R)} \sqrt{\frac{\pi}{2x \left|\frac{\dee^2 \Phi}{\dee v^2}(v_2,R)\right|}} e^{\frac{i\pi}{4} \sgn\left(\frac{\dee^2 \Phi}{\dee v^2}(v_2,R)\right)}.
%\end{multline*}
%We have that
%\[\Psi(v_1) = -1, \qquad \Psi(v_2) = 1, \qquad \Phi(v_1,R) = -1, \qquad \Phi(v_2,R) = 1,\]
%and
%\[\frac{\dee^2 \Phi}{\dee v^2}(v_1,R) = \frac{2 \tanh \frac{R}{2}}{R}, \qquad \frac{\dee^2 \Phi}{\dee v^2}(v_2,R) = -\frac{2 \tanh \frac{R}{2}}{R}.\]
%It follows that
%\[\int_{-\pi/2}^{\pi/2} \Psi(v) e^{ix \Phi(v,R)} \, dv \sim i \sqrt{\frac{2\pi}{x} \frac{2 \tanh \frac{R}{2}}{R}} \sin\left(x - \frac{\pi}{4}\right),\]
%which gives the identity.
Using stationary phase, with the two critical points being the endpoints $\pm \pi/2$, we find that there exists some $R_0 > 0$ such that
\[\sup_{R \in (0,R_0)} \left|x^{3/2} h(R,x) - 2 \sqrt{\frac{2}{\pi} \frac{R}{\sinh R}}\sin\left(x - \frac{\pi}{4}\right)\right| \ll \frac{1}{x}.\qedhere\]
\end{proof}

For a function $k : [0,\infty) \to \C$, we may form the automorphic kernel
\[K(z,w) \defeq \sum_{\gamma \in \Gamma} k(\gamma z,w),\]
which is $\Gamma$-invariant in both variables. When $k(u) = k_R(u)$, we write $K(z,w) = K_R(z,w)$.

\begin{lemma}\label{heigenlemma}
If $f : \Gamma \backslash \Hb \to \C$ is an eigenfunction of the Laplacian with eigenvalue $1/4 + t_f^2$, then
\[\frac{1}{\vol\left(B_R\right))} \int_{B_R(w)} f(z) \, d\mu(z) = \langle f, K_R(\cdot,w)\rangle = h_R(t_f)  f(w).\]
\end{lemma}

\begin{proof}
This follows from \cite[Theorem 1.14]{Iwa}. Note that there it is assumed that not only is $k(u)$ compactly supported, but that it is smooth; this, however, is not essential to the proof. Instead, we merely require that $k(z,w)$ be twice differentiable in both variables $\mu$-almost everywhere.
\end{proof}

\subsection{Proof of \texorpdfstring{\hyperref[Planckscalethm]{Theorem \ref*{Planckscalethm}}}{Theorem \ref{Planckscalethm}}}

\begin{proposition}[{\cite[Theorem 1]{Mil}}]\label{Milicevicthm}
For every fixed Heegner point $w \in \Hb$,
\[|g(w)| = \Omega\left(\exp\left(\sqrt{\frac{\log t_g}{\log \log t_g}} \left(1 + O\left(\frac{\log \log \log t_g}{\log \log t_g}\right)\right)\right)\right)\]
as $t_g$ tends to infinity. 
\end{proposition}

\begin{proof}[Proof of \texorpdfstring{\hyperref[Planckscalethm]{Theorem \ref*{Planckscalethm}}}{Theorem \ref{Planckscalethm}}]
For $g \in \BB_0(\Gamma)$,
\[\frac{1}{\vol\left(B_R\right)} \int_{B_R(w)} g(z) \, d\mu(z) = \int_{\Gamma \backslash \Hb} K_R(z,w) g(z) \, d\mu(z) = h_R(t_g) g(w).\]
It follows by the Cauchy--Schwarz inequality that
\[\left|h_R\left(t_g\right)\right|^2 |g(w)|^2 \leq \frac{1}{\vol\left(B_R\right)} \int_{B_R(w)} |g(z)|^2 \, d\mu(z).\]
\hyperref[Planckscalethm]{Theorem \ref*{Planckscalethm}} then follows from \hyperref[hR(t)asymplemma]{Lemma \ref*{hR(t)asymplemma}} and \hyperref[Milicevicthm]{Proposition \ref*{Milicevicthm}}.
\end{proof}

\begin{remark}
\hyperref[Planckscalethm]{Theorem \ref*{Planckscalethm}} also holds for Maa\ss{} newforms $g \in \BB_0^{\ast}(\Gamma_0(q))$ for any $q > 1$, for \hyperref[Milicevicthm]{Proposition \ref*{Milicevicthm}} is proved in this generality (and in fact in even further generality).
\end{remark}

\begin{remark}
Since it is conjectured that $\max_{w \in K} |g(w)| \ll_{K,\e} t_g^{\e}$ for every compact subset $K$ of $\Gamma \backslash \Hb$, we cannot expect any significant improvement to \hyperref[Planckscalethm]{Theorem \ref*{Planckscalethm}} via this line of reasoning.
\end{remark}

\section{Equidistribution in Almost Every Shrinking Ball}

\subsection{Proof of Conditional Results}

In this section, we prove the following.

\begin{proposition}\label{shrinkingsetcondprop}
Let $g \in \BB_0(\Gamma)$ be a Hecke--Maa\ss{} eigenform normalised such that $\langle g, g \rangle = 1$. For $R > 0$, let
\[\Var(g;R) \defeq \int_{\Gamma \backslash \Hb} \left(\frac{1}{\vol\left(B_R\right)} \int_{B_R(w)} |g(z)|^2 \, d\mu(z) - \frac{1}{\vol\left(\Gamma \backslash \Hb\right)}\right)^2 \, d\mu(w).\]
Assume the generalised Lindel\"{o}f hypothesis, and suppose that $R \asymp t_g^{-\delta}$ for some $\delta > 0$. Then for $0 < \delta < 1$,
\[\Var(g;R) \ll_{\e} t_g^{-(1 - \delta) + \e}\]
as $t_g$ tends to infinity, while for $\delta > 1$,
\[\Var(g;R) \sim \frac{2}{\vol\left(\Gamma \backslash \Hb\right)} = \frac{6}{\pi}\]
as $t_g$ tends to infinity.
\end{proposition}

\hyperref[Maassalmosteverythm]{Theorem \ref*{Maassalmosteverythm}} then follows directly via Chebyshev's inequality. Our starting point towards proving \hyperref[shrinkingsetcondprop]{Proposition \ref*{shrinkingsetcondprop}} is the following spectral expansion of $\Var(g;R)$.

\begin{proposition}\label{Var(g;R)cuspspectralprop}
Let $g \in \BB_0(\Gamma)$ be a Hecke--Maa\ss{} eigenform normalised such that $\langle g, g \rangle = 1$. Then $\Var(g;R)$ is equal to
\[\sum_{f \in \BB_0\left(\Gamma\right)} \left|h_R(t_f)\right|^2 \left|\left\langle |g|^2, f\right\rangle\right|^2 + \frac{1}{4 \pi} \int_{-\infty}^{\infty} \left|h_R(t)\right|^2 \left|\left\langle |g|^2, E\left(\cdot, \frac{1}{2} + it\right) \right\rangle\right|^2 \, dt,\]
where
\[h_R(t) \defeq \frac{R}{\pi \sinh \frac{R}{2}} \int_{-1}^{1} \sqrt{1 - \left(\frac{\sinh \frac{Rr}{2}}{\sinh \frac{R}{2}}\right)^2} e^{iRrt} \, dr.\]
\end{proposition}

\begin{proof}
Via \hyperref[spectralexpansionlemma]{Lemmata \ref*{spectralexpansionlemma}} (namely Parseval's identity) and \ref{heigenlemma}, $\langle |g|^2, K_R(\cdot,w) \rangle$ is equal to
\begin{multline*}
\frac{\langle g, g \rangle}{\vol\left(\Gamma \backslash \Hb\right)} + \sum_{f \in \BB_0\left(\Gamma\right)} h_R(t_f) f(w) \left\langle |g|^2, f\right\rangle	\\
+ \frac{1}{4\pi} \int_{-\infty}^{\infty} h_R(t) E\left(w, \frac{1}{2} + it\right) \left\langle |g|^2, E\left(\cdot, \frac{1}{2} + it\right) \right\rangle \, dt.
\end{multline*}
Upon squaring and integrating over $w$, we obtain the desired identity.
\end{proof}

\begin{proof}[Proof of \texorpdfstring{\hyperref[shrinkingsetcondprop]{Proposition \ref*{shrinkingsetcondprop}}}{Proposition \ref{shrinkingsetcondprop}} for $0 < \delta < 1$]
We use \hyperref[Var(g;R)cuspspectralprop]{Propositions \ref*{Var(g;R)cuspspectralprop}} and \ref{innerproductcuspasympprop} and \hyperref[hR(t)asymplemma]{Lemmata \ref*{hR(t)asymplemma}} and \ref{Gammafactorslemma}. We then divide the spectral expansion in \hyperref[Var(g;R)cuspspectralprop]{Proposition \ref*{Var(g;R)cuspspectralprop}} into various ranges.

Just as in \hyperref[spectralranglesL4sect]{Section \ref*{spectralranglesL4sect}}, there are two main ranges of the continuous spectrum to consider:
\begin{itemize}
\item the initial range $0 \leq |t| < 2t_g + t_g^{\delta}$, and
\item the tail range $|t| > 2t_g + t_g^{\delta}$.
\end{itemize}
The division of the cuspidal spectrum into parts depends on $\delta$. When $R \asymp t_g^{-\delta}$ with $0 < \delta < 1$, the ranges are:
\begin{itemize}
\item the short initial range $0 < t_f \leq t_g^{\delta}$,
\item the polynomial decay range $t_g^{\delta} < t_f < 2t_g + t_g^{1 - \delta}$,
\item the tail range $t_f \geq 2t_g + t_g^{1 - \delta}$.
\end{itemize}

Thus $\Var(g;R)$ is bounded by a constant multiple dependent on $\e$ of
\begin{multline*}
t_g^{-1 + \e} \sum_{0 < t_f \leq t_g^{\delta}} \frac{L\left(\frac{1}{2}, f\right) L\left(\frac{1}{2}, \sym^2 g \otimes f\right)}{t_f L\left(1, \sym^2 f\right)}	\\
+ t_g^{3\delta - \frac{1}{2} + \e} \sum_{t_g^{\delta} < t_f < 2t_g + t_g^{1 - \delta}} \frac{L\left(\frac{1}{2}, f\right) L\left(\frac{1}{2}, \sym^2 g \otimes f\right)}{t_f^4 (1 + |2t_g - t_f|)^{1/2} L\left(1, \sym^2 f\right)}	\\
+ t_g^{3\delta + \e} \sum_{t_f \geq 2t_g + t_g^{1 - \delta}} e^{-\pi (t_f - 2t_g)} \frac{L\left(\frac{1}{2}, f\right) L\left(\frac{1}{2}, \sym^2 g \otimes f\right)}{t_f^{\frac{9}{2}} (1 + t_f - 2t_g)^{1/2} L\left(1, \sym^2 f\right)}	\\
+ t_g^{-\frac{1}{2} + \e} \int_{0}^{2t_g + t_g^{\delta}} \frac{\left|L\left(\frac{1}{2} + it\right) L\left(\frac{1}{2} + it, \sym^2 g\right)\right|^2}{(1 + t) (1 + |2t_g - t|)^{1/2} |\zeta(1 + 2it)|^2} \, dt	\\
+ t_g^{-\frac{1}{2} + \e} \int_{2t_g + t_g^{\delta}}^{\infty} e^{-\pi (t - 2t_g)} \frac{\left|L\left(\frac{1}{2} + it\right) L\left(\frac{1}{2} + it, \sym^2 g\right)\right|^2}{(1 + t) (1 + |2t_g - t|)^{1/2} |\zeta(1 + 2it)|^2} \, dt.
\end{multline*}
\begin{itemize}
\item From \cite[Lemma 2.1]{BK17b}, the initial and tail ranges of the continuous spectrum are bounded by $t_g^{-1 + \e}$.
\item The convexity bounds for $L(1/2, f)$ and $L(1/2, \sym^2 g \otimes f)$ show that the tail range of the cuspidal spectrum is rapidly decaying.
\item For the other two ranges, the generalised Lindel\"{o}f hypothesis implies that the product of these two $L$-functions is bounded by a constant multiple dependent on $\e$ of $t_g^{\e}$, and then the Weyl law for $\Gamma \backslash \Hb$ and partial summation imply that the contribution of the cuspidal spectrum is bounded by $t_g^{\delta - 1 + \e}$.
\end{itemize}
This completes the proof.
\end{proof}

\begin{proof}[Proof of \texorpdfstring{\hyperref[shrinkingsetcondprop]{Proposition \ref*{shrinkingsetcondprop}}}{Proposition \ref{shrinkingsetcondprop}} for $\delta > 1$]
In this case, the division of the cuspidal spectrum into parts involves an additional range, and there is a dependence on an small fixed parameter $\delta' > 0$:
\begin{itemize}
\item the short initial range $0 < t_f \leq t_g^{1 - \delta'}$, which once again is bounded by $t_g^{-\delta'/2 + \e}$ via the generalised Lindel\"{o}f hypothesis,
\item the bulk range $t_g^{1 - \delta'} < t_f < 2t_g - t_g^{1 - \delta'}$, which is asymptotic to $6/\pi$ from the proof of \cite[Proposition 2.2]{BK17b},
\item the short transition range $2t_g - t_g^{1 - \delta'} \leq t_f \leq 2t_g + t_g^{1 - \delta'}$, again bounded by $t_g^{-\delta'/2 + \e}$, and
\item the tail range $t_f > 2t_g + t_g^{1 - \delta'}$, which is negligible.
\end{itemize}
This completes the proof.
\end{proof}

\begin{remark}
Just as with \hyperref[Planckscalethm]{Theorem \ref*{Planckscalethm}}, the bound $\Var(g;R) \ll_{\e} t_g^{-(1 - \delta) + \e}$ for $R \asymp t_g^{-\delta}$ with $0 < \delta < 1$ in \hyperref[shrinkingsetcondprop]{Proposition \ref*{shrinkingsetcondprop}} also holds for Maa\ss{} newforms $g \in \BB_0^{\ast}(\Gamma_0(q))$ for any $q > 1$. Indeed, \cite[Theorem 15.5]{IK} gives the spectral decomposition of $L^2(\Gamma_0(q) \backslash \Hb)$, though there are Eisenstein series corresponding to each cusp and the orthonormal basis of Maa\ss{} cusp forms are no longer necessarily Hecke--Maa\ss{} eigenforms. Nonetheless, Blomer and Mili\'{c}evi\'{c} have given an orthonormal basis of $\BB_0(\Gamma_0(q))$ involving linear combinations of oldforms and newforms \cite[Lemma 9]{BM}, and a similar basis exists for the space of Eisenstein series \cite{You17b}, and these can be coupled with the work of Hu on the Watson--Ichino formula in this generality \cite{Hu}.
\end{remark}

\begin{remark}
In fact, the method of proof of \cite[Proposition 2.2]{BK17b} together with \hyperref[hR(t)asymplemma]{Lemma \ref*{hR(t)asymplemma}} show that if $R \sim (C t_g)^{-1}$ for some positive constant $C$, then
\[\Var(g;R) \sim \frac{12 C}{\pi^2} \int_{0}^{1} \frac{J_1\left(\frac{2t}{C}\right)}{t \sqrt{1 - t^2}} \, dt = \frac{6}{\pi} \left(J_0\left(\frac{1}{C}\right)^2 + J_1\left(\frac{1}{C}\right)^2\right)\]
by \cite[(8.473.1) and (6.552.4)]{GR}, which converges to $6/\pi$ as $C$ tends to infinity.
\end{remark}

\subsection{Proof of Unconditional Results}

We first sketch how to prove \hyperref[improvedYoungEisenthm]{Theorem \ref*{improvedYoungEisenthm}}.

\begin{proof}[Proof of \texorpdfstring{\hyperref[improvedYoungEisenthm]{Theorem \ref*{improvedYoungEisenthm}}}{Theorem \ref{improvedYoungEisenthm}}]
In \cite{You16}, after \cite[(4.24)]{You16}, we use \hyperref[shortinitialrangelemma]{Lemma \ref*{shortinitialrangelemma}} instead of the subconvexity bound $L(1/2 + it,f) \ll_{\e} (t_f + t)^{1/3 + \e}$. Using this, the right-hand side of \cite[(4.26)]{You16} is improved to $T^{-1/6 + \e} \|\phi\|_2$, which yields the result.
\end{proof}

Next, we cover the proof of the following, from which \hyperref[Maassalmosteverythm]{Theorem \ref*{Maassalmosteverythm}} will be derived.

\begin{proposition}\label{shrinkingsetuncondprop}
Let $g(z) = E\left(z, 1/2 + it_g\right)$. For $R > 0$, let
\[\Var(g;R) \defeq \int_{\Gamma \backslash \Hb} \left(\frac{1}{\vol\left(B_R\right)} \int_{B_R(w)} |g(z)|^2 \, d\mu(z) - C(g;R;w)\right)^2 \, d\mu(w),\]
where $C(g;R;w)$ is given by \eqref{C(g;R;w)}. Suppose that $R \asymp t_g^{-\delta}$ for some $0 < \delta < 1$. Then
\[\Var(g;R) \ll_{\e} t_g^{-\min\left\{\frac{5}{7} (1 - \delta), \frac{1}{6} \right\} + \e}.\]
\end{proposition}

To begin, we wish to calculate
\[\frac{1}{\vol(B_R)} \int_{B_R(w)} |g(z)|^2 \, d\mu(z),\]
where $g(z) = E\left(z, 1/2 + it_g\right)$. However, we cannot use Parseval's identity because $|g|^2 \notin L^2\left(\Gamma \backslash \Hb\right)$. Instead, we replace $|g(z)|^2$ with $E(z,s_1) E(z,s_2)$ and subtract away a linear combination of Eisenstein series $\EE$ such that the resulting function is square-integrable. After applying Parseval's identity, we finally send $s_1$ to $1/2 + it_g$ and $s_2$ to $1/2 - it_g$.

\begin{lemma}[{Cf.~\cite[Lemma 4.1]{You16}}]\label{FEElemma}
For $1/2 < \Re(s_1), \Re(s_2) < 3/4$,
\[\frac{1}{\vol(B_R)} \int_{B_R(w)} E(z,s_1) E(z,s_2) \, d\mu(z)\]
is equal to
\begin{multline*}
h_R\left(i\left(s_1 + s_2 - \frac{1}{2}\right)\right) E(w,s_1 + s_2)	\\
+ h_R\left(i\left(\frac{1}{2} - s_1 + s_2\right)\right) \frac{\Lambda(2 - 2s_1)}{\Lambda(2s_1)} E(w,1 - s_1 + s_2)	\\
+ h_R\left(i\left(\frac{1}{2} + s_1 - s_2\right)\right) \frac{\Lambda(2 - 2s_2)}{\Lambda(2s_2)} E(w,1 + s_1 - s_2)	\\
+ h_R\left(i\left(\frac{3}{2} - s_1 - s_2\right)\right) \frac{\Lambda(2 - 2s_1) \Lambda(2 - 2s_2)}{\Lambda(2s_1) \Lambda(2s_2)} E(w,2 - s_1 - s_2)	\\
+ \sum_{f \in \BB_0(\Gamma)} h_R(t_f) f(w) \left\langle E(\cdot,s_1) E(\cdot,s_2), f \right\rangle	\\
+ \frac{1}{4\pi} \int_{-\infty}^{\infty} h_R(t) E\left(w, \frac{1}{2} + it\right) \left\langle E(\cdot,s_1) E(\cdot,s_2), E\left(\cdot, \frac{1}{2} + it\right) \right\rangle_{\reg} \, dt.
\end{multline*}
\end{lemma}

\begin{proof}
Let $F(z) \defeq E(z,s_1) E(z,s_2)$ and let
\begin{multline*}
\EE(z) \defeq E(z,s_1 + s_2) + \frac{\Lambda(2 - 2s_1)}{\Lambda(2s_1)} E(z,1 - s_1 + s_2)	\\
+ \frac{\Lambda(2 - 2s_2)}{\Lambda(2s_2)} E(z,1 + s_1 - s_2) + \frac{\Lambda(2 - 2s_1) \Lambda(2 - 2s_2)}{\Lambda(2s_1) \Lambda(2s_2)} E(z,2 - s_1 - s_2).
\end{multline*}
%The constant term of $\EE(z)$ is
%\begin{multline*}
%y^{s_1 + s_2} + \varphi(s_1) y^{1 - s_1 + s_2} + \varphi(s_2) y^{1 + s_1 - s_2} + \varphi(s_1) varphi(s_2) y^{2 - s_1 - s_2}	\\
% + \varphi(s_1 + s_2) y^{1 - s_1 - s_2} + \varphi(s_1) \varphi(1 - s_1 + s_2) y^{s_1 - s_2} + \varphi(s_2) \varphi(1 + s_1 - s_2) y^{s_2 - s_1} + \varphi(s_1) \varphi(s_2) \varphi(2 - s_1 - s_2) y^{s_1 + s_2 - 1}.
%\end{multline*}
Since the constant term of $F(z)$ is
\[y^{s_1 + s_2} + \frac{\Lambda(2 - 2s_1)}{\Lambda(2s_1)} y^{1 - s_1 + s_2} + \frac{\Lambda(2 - 2s_2)}{\Lambda(2s_2)} y^{1 + s_1 - s_2} + \frac{\Lambda(2 - 2s_1) \Lambda(2 - 2s_2)}{\Lambda(2s_1) \Lambda(2s_2)} y^{2 - s_1 - s_2},\]
we have that $F(z) - \EE(z) = O(y^{1/2 - \delta})$ for some $\delta > 0$ at the cusp at infinity, and consequently $F - \EE \in L^2\left(\Gamma \backslash \Hb\right)$. \hyperref[spectralexpansionlemma]{Lemmata \ref*{spectralexpansionlemma}} (namely Parseval's identity) and \ref{heigenlemma} then imply that
\begin{multline*}
\langle F - \EE, K_R(\cdot,w) \rangle = \frac{\langle F - \EE, 1 \rangle}{\vol\left(\Gamma \backslash \Hb\right)} + \sum_{f \in \BB_0(\Gamma)} h_R(t_f) f(w) \langle F - \EE, f \rangle	\\
+ \frac{1}{4\pi} \int_{-\infty}^{\infty} h_R(t) E\left(w, \frac{1}{2} + it\right) \left\langle F - \EE, E\left(\cdot, \frac{1}{2} + it\right) \right\rangle \, dt.
\end{multline*}

The left-hand side is equal to $\langle F, K_R(\cdot,w) \rangle - \langle \EE, K_R(\cdot,w) \rangle$, and \hyperref[heigenlemma]{Lemma \ref*{heigenlemma}} allows us to calculate $\langle \EE, K_R(\cdot,w) \rangle$ explicitly. On the right-hand side, the inner product $\langle \EE, f \rangle$ vanishes whenever $f \in \BB_0(\Gamma)$, being the linear combination of inner products of Eisenstein series with a cusp form, and similarly $\langle F - \EE, 1 \rangle$ vanishes via \cite[Equation (36) and Section 2]{Zag}. Finally, we claim that the inner product $\langle F - \EE, E\left(\cdot, \frac{1}{2} + it\right) \rangle$ is equal to
\[\frac{\Lambda\left(s_2 - s_1 + \frac{1}{2} + it\right) \Lambda\left(s_1 + s_2 - \frac{1}{2} + it\right) \Lambda\left(s_2 - s_1 + \frac{1}{2} - it\right) \Lambda\left(s_1 + s_2 - \frac{1}{2} - it\right)}{\Lambda\left(2s_1\right) \Lambda\left(2s_2\right) \Lambda\left(1 - 2it\right)}.\]
Indeed, we may add and subtract a linear combination of Eisenstein series $\EE'$ such that both $F E\left(\cdot, 1/2 - it\right) - \EE'$ and $\EE E\left(\cdot, 1/2 - it\right) - \EE'$ are integrable. Then the integral of $\EE E\left(\cdot, 1/2 + it\right) - \EE'$ vanishes via \cite[Equation (36) and Section 2]{Zag}, and the integral of $F E\left(\cdot, 1/2 + it\right) - \EE'$ is equal to the desired product of completed zeta functions via \cite[Equation (44)]{Zag}.
\end{proof}

We now define
\begin{equation}\label{D(g;w)}
D(g;w) \defeq \frac{2}{\vol\left(\Gamma \backslash \Hb\right)} \left(2\Re\left(\frac{\Lambda'}{\Lambda} \left(1 + 2it_g\right)\right) + 2\gamma_0 - \frac{12 \zeta'(2)}{\pi^2} - \log \left|4 \Im(w) \eta(w)^4\right|\right).
\end{equation}
Here $\gamma_0$ is the Euler--Mascheroni constant and
\[\eta(w) \defeq e\left(\frac{w}{24}\right) \prod_{m = 1}^{\infty} \left(1 - e(mw)\right)\]
denotes the Dedekind eta function; note that $\Im(w)^6 \eta(w)^{24}$ is a Maa\ss{} cusp form of weight $12$ and level $1$ that is nonvanishing outside the single cusp of $\Gamma \backslash \Hb$. That $D(g;w)$ is, in some sense, the ``true'' average of $|E(z,1/2 + it_g)|^2$ on compact sets, rather than
\[\frac{\log \left(\frac{1}{4} + t_g^2\right)}{\vol\left(\Gamma \backslash \Hb\right)},\]
has previously been observed by Young \cite[Section 4.2]{You16} and also Hejhal and Rackner \cite[p.~300]{HejRa}, though in the latter case, their expression does not include the Dedekind eta function.

\begin{proof}[Proof of \texorpdfstring{\hyperref[D(g;w)compactlemma]{Lemma \ref*{D(g;w)compactlemma}}}{Lemma \ref{D(g;w)compactlemma}}]
This follows from \eqref{varphi'varphi}, \eqref{Gamma'Gamma}, and \eqref{zeta'zeta}, together with the fact that $\Im(w)^6 \eta(w)^{24}$ is nonvanishing in $K$.
\end{proof}

We define
\begin{equation}\label{C(g;R;w)}
C(g;R;w) \defeq D(g;w) + \frac{2 i h_R'\left(\frac{i}{2}\right)}{\vol\left(\Gamma \backslash \Hb\right)} + 2\Re\left(h_R\left(2t_g + \frac{i}{2}\right) \frac{\Lambda(1 - 2it_g)}{\Lambda(1 + 2it_g)} E\left(w,1 - 2it_g\right)\right).
\end{equation}

\begin{lemma}\label{smallscaleEisenlemma}
Let $g(z) = E\left(z, 1/2 + it_g\right)$. Then
\begin{multline*}
\frac{1}{\vol(B_R)} \int_{B_R(w)} |g(z)|^2 \, d\mu(z) = C(g;R;w) + \sum_{f \in \BB_0(\Gamma)} h_R(t_f) f(w) \left\langle |g|^2, f \right\rangle	\\
+ \frac{1}{4\pi} \int_{-\infty}^{\infty} h_R(t) E\left(w, \frac{1}{2} + it\right) \left\langle |g|^2, E\left(\cdot, \frac{1}{2} + it\right) \right\rangle_{\reg} \, dt.
\end{multline*}
\end{lemma}

\begin{proof}
This follows from \hyperref[FEElemma]{Lemma \ref*{FEElemma}} upon taking $s_1 = 1/2 + it_g + \e$ and $s_2 = 1/2 - it_g + \e$ and using the expansions
\begin{align*}
h_R\left(i\left(\frac{1}{2} + 2\e\right)\right) & = 1 + 2i h_R'\left(\frac{i}{2}\right) \e + O(\e^2),	\\
\frac{\Lambda(1 - 2it_g - 2\e) \Lambda(1 + 2it_g - 2\e)}{\Lambda(1 + 2it_g + 2\e) \Lambda(1 - 2it_g + 2\e)} & = 1 - 8 \Re \left(\frac{\Lambda'}{\Lambda} \left(1 + 2it_g\right)\right) \e + O(\e^2),	\\
\vol\left(\Gamma \backslash \Hb\right) E(w,1 + 2\e) & = \frac{1}{2\e} + 2\gamma_0 - \log \left|4 \Im(w) \eta(w)^4\right| - \frac{12 \zeta'(2)}{\pi^2} + O(\e),
\end{align*}
where the last line is the Kronecker limit formula.
%This is (22.69) of IK.
\end{proof}

With this in hand, we can finally give the spectral expansion of $\Var(g;R)$.

\begin{proposition}\label{Var(g;R)Eisenspectralprop}
Let $g(z) = E\left(z, 1/2 + it_g\right)$. Then $\Var(g;R)$ is equal to
\[\sum_{f \in \BB_0(\Gamma)} \left|h_R(t_f)\right|^2  \left|\left\langle |g|^2, f \right\rangle \right|^2 + \frac{1}{4\pi} \int_{-\infty}^{\infty} \left|h_R(t)\right|^2 \left|\left\langle |g|^2, E\left(\cdot, \frac{1}{2} + it\right) \right\rangle_{\reg} \right|^2 \, dt.\]
\end{proposition}

\begin{proof}
This follows directly from \hyperref[smallscaleEisenlemma]{Lemma \ref*{smallscaleEisenlemma}} after an application of Parseval's identity in \hyperref[spectralexpansionlemma]{Lemma \ref*{spectralexpansionlemma}}.
\end{proof}

\begin{proof}[Proof of \texorpdfstring{\hyperref[shrinkingsetuncondprop]{Proposition \ref*{shrinkingsetuncondprop}}}{Proposition \ref{shrinkingsetuncondprop}}]
We use \hyperref[Var(g;R)Eisenspectralprop]{Propositions \ref*{Var(g;R)Eisenspectralprop}} and \ref{innerproductEisenasympprop} and \hyperref[hR(t)asymplemma]{Lemmata \ref*{hR(t)asymplemma}} and \ref{Gammafactorslemma}. We then divide the spectral expansion in \hyperref[Var(g;R)Eisenspectralprop]{Proposition \ref*{Var(g;R)Eisenspectralprop}} into various ranges.

The two ranges of the continuous spectrum are:
\begin{itemize}
\item the initial range $0 \leq |t| < 2t_g + t_g^{\delta}$, and
\item the tail range $|t| > 2t_g + t_g^{\delta}$.
\end{itemize}
The cuspidal spectrum can be broken into five ranges, which depend on a small fixed parameter $0 < \delta' < 1 - \delta$:
\begin{itemize}
\item the short initial range $0 < t_f \leq t_g^{\delta}$,
\item the short initial polynomial decay range $t_g^{\delta} < t_f < t_g^{1 - \delta'}$,
\item the bulk polynomial decay range $t_g^{1 - \delta'} \leq t_f \leq 2t_g - t_g^{\delta}$,
\item the short transition polynomial decay range $2t_g - t_g^{\delta} < t_f < 2t_g + t_g^{\delta}$,
\item the tail range $t_f \geq 2t_g + t_g^{\delta}$.
\end{itemize}

Thus $\Var(g;R)$ is bounded by a constant multiple dependent on $\e$ of
\begin{multline*}
t_g^{-1 + \e} \sum_{0 < t_f \leq t_g^{\delta}} \frac{L\left(\frac{1}{2}, f\right)^2 \left|L\left(\frac{1}{2} + 2it_g, f\right) \right|^2}{t_f L\left(1, \sym^2 f\right)}	\\
+ t_g^{3\delta - 1 + \e} \sum_{t_g^{\delta} < t_f < t_g^{1 - \delta'}} \frac{L\left(\frac{1}{2}, f\right)^2 \left|L\left(\frac{1}{2} + 2it_g, f\right) \right|^2}{t_f^4 L\left(1, \sym^2 f\right)}	\\
+ t_g^{3\delta - \frac{1}{2} + \e} \sum_{t_g^{1 - \delta'} \leq t_f \leq 2t_g - t_g^{\delta}} \frac{L\left(\frac{1}{2}, f\right)^2 \left|L\left(\frac{1}{2} + 2it_g, f\right) \right|^2}{t_f^4 (1 + |2t_g - t_f|)^{1/2} L\left(1, \sym^2 f\right)}	\\
+ t_g^{3\delta - \frac{9}{2} + \e} \sum_{2t_g - t_g^{\delta} < t_f < 2t_g + t_g^{\delta}} \frac{L\left(\frac{1}{2}, f\right)^2 \left|L\left(\frac{1}{2} + 2it_g, f\right) \right|^2}{(1 + |2t_g - t_f|)^{1/2} L\left(1, \sym^2 f\right)}	\\
+ t_g^{3\delta + \e} \sum_{t_f \geq 2t_g + t_g^{\delta}} e^{-\pi (t_f - 2t_g)} \frac{L\left(\frac{1}{2}, f\right)^2 \left|L\left(\frac{1}{2} + 2it_g, f\right) \right|^2}{t_f^{\frac{9}{2}} (1 + t_f - 2t_g)^{1/2} L\left(1, \sym^2 f\right)}	\\
+ t_g^{-\frac{1}{2} + \e} \int_{0}^{2t_g + t_g^{\delta}} \frac{\left| \zeta\left(\frac{1}{2} + i(2t_g + t)\right) \zeta\left(\frac{1}{2} + it\right)^2 \zeta\left(\frac{1}{2} + i(2t_g - t)\right) \right|^2}{(1 + t) (1 + |2t_g - t|)^{1/2} \left|\zeta\left(1 - 2it\right)\right|^2} \, dt	\\
+ t_g^{-\frac{1}{2} + \e} \int_{2t_g + t_g^{\delta}}^{\infty} e^{-\pi (t - 2t_g)} \frac{\left| \zeta\left(\frac{1}{2} + i(2t_g + t)\right) \zeta\left(\frac{1}{2} + it\right)^2 \zeta\left(\frac{1}{2} + i(2t_g - t)\right) \right|^2}{(1 + t) (1 + |2t_g - t|)^{1/2} \left|\zeta\left(1 - 2it\right)\right|^2} \, dt.
\end{multline*}
The continuous spectrum is readily dealt with:
\begin{itemize}
\item From \cite[Proposition 3.4]{Spi} and \cite[Theorem 5]{Bou}, the initial and tail ranges of the continuous spectrum are bounded by a constant multiple dependent on $\e$ of $t_g^{-\frac{13}{84} + \e}$.
\end{itemize}
For the cuspidal spectrum, we have the following:
\begin{itemize}
\item The convexity bounds for $L(1/2, f)$ and $L(1/2 + 2it_g, f)$ show that the tail range is rapidly decaying.
\item The short initial range is bounded by a constant multiple dependent on $\e$ of $t_g^{-\min\left\{1 - \delta, 1/6\right\} + \e}$ upon dividing into dyadic intervals and applying \hyperref[shortinitialrangelemma]{Lemma \ref*{shortinitialrangelemma}}.
\item The same method bounds the short initial polynomial decay range by $t_g^{-\min\left\{\delta', 1/6\right\} + \e}$.
\item For the bulk polynomial decay range, we divide into dyadic intervals and use \hyperref[largesieveinitiallemma]{Lemma \ref*{largesieveinitiallemma}}, which shows that this range is bounded by  $t_g^{-\frac{5}{2}(1 - \delta - \delta') + \e}$.
\item We divide the short transition polynomial decay range into intervals of length $t_g^{1/3}$, use the Cauchy--Schwarz inequality, and apply \hyperref[shorttransitionrangelemma]{Lemma \ref*{shorttransitionrangelemma}}, which gives the bound $t_g^{-\frac{7}{2} (1 - \delta) + \e}$.
\end{itemize}
\hyperref[shrinkingsetuncondprop]{Proposition \ref*{shrinkingsetuncondprop}} is proven upon taking $\delta' = \frac{5}{7} (1 - \delta)$.
\end{proof}

\begin{proof}[Proof of \texorpdfstring{\hyperref[Eisenalmosteverythm]{Theorem \ref*{Eisenalmosteverythm}}}{Theorem \ref{Eisenalmosteverythm}}]
By Chebyshev's inequality and \hyperref[shrinkingsetuncondprop]{Proposition \ref*{shrinkingsetuncondprop}},
\begin{multline*}
\vol\left(\left\{w \in \Gamma \backslash \Hb : \left|\frac{1}{\vol(B_R)} \int_{B_R(w)} |g(z)|^2 \, d\mu(z) - C(g;R;w)\right| > c\right\}\right)	\\
\ll_{\e} \frac{t_g^{-\min\left\{\frac{5}{7} (1 - \delta), \frac{1}{6} \right\} + \e}}{c^2}.
\end{multline*}

Again by Chebyshev's inequality,
\begin{multline*}
\vol\left(\left\{w \in \Gamma \backslash \Hb : \left|h_R\left(2t_g + \frac{i}{2}\right) E\left(w,1 - 2it_g\right)\right| > c\right\}\right)	\\
\leq \vol\left(\left\{w \in \Gamma \backslash \Hb : \Im(w) > T\right\}\right) + \frac{\left|h_R\left(2t_g + \frac{i}{2}\right)\right|^2}{c^2} \int_{\Gamma \backslash \Hb} \left|\Lambda^T E\left(w, 1 - 2it_g\right)\right|^2 \, d\mu(z)
\end{multline*}
for any $T \geq 1$, which, by the Maa\ss{}--Selberg relation \eqref{MaassSelberg}, is equal to
\[\frac{1}{T} + \frac{\left|h_R\left(2t_g + \frac{i}{2}\right)\right|^2}{c^2}\left(T + 2\Re\left(\frac{\Lambda(1 - 4it_g)}{\Lambda(2 - 4it_g)} \frac{T^{4it_g}}{4it_g}\right) + \left|\frac{\Lambda(1 - 4it_g)}{\Lambda(2 - 4it_g)}\right|^2 \frac{1}{T}\right).\]
Using stationary phase as in the proof of \hyperref[hR(t)asymplemma]{Lemma \ref*{hR(t)asymplemma}}, or alternatively using \cite[Lemma 2.4]{Cha96}, we have that $|h_R\left(2t_g + \frac{i}{2}\right)|^2 \ll t_g^{-3(1 - \delta)}$, while Stirling's approximation implies that
\[\frac{\Lambda(1 - 2it_g)}{\Lambda(2 - 4it_g)} \ll_{\e} t_g^{-\frac{1}{2} + \e}.\]

Next, we note that
\[ih_R'\left(\frac{i}{2}\right) = \frac{R^2}{\pi} \int_{-1}^{1} r \sqrt{1 - \left(\frac{\sinh \frac{Rr}{2}}{\sinh \frac{R}{2}}\right)^2} \frac{\sinh \frac{Rr}{2}}{\sinh \frac{R}{2}} \, dr \sim \frac{R^2}{8} \asymp t_g^{-2\delta},\]
so if $c \gg_{\e} t_g^{-2\delta + \e}$, then for all sufficiently large $t_g$,
\[\left| \frac{2 i h_R'\left(\frac{i}{2}\right)}{\vol\left(\Gamma \backslash \Hb\right)} \right| < c.\]

So piecing everything together, we find that if $c \gg_{\e} t_g^{-2\delta + \e}$,
\begin{multline*}
\vol\left(\left\{w \in \Gamma \backslash \Hb : \left|\frac{1}{\vol(B_R)} \int_{B_R(w)} |g(z)|^2 \, d\mu(z) - D(g;w)\right| > c\right\}\right)	\\
\ll_{\e} \frac{t_g^{-\frac{5}{7} (1 - \delta) + \e}}{c^2} + \frac{t_g^{-\frac{1}{6} + \e}}{c^2} + \frac{1}{T} + \frac{t_g^{-3(1 - \delta)} T}{c^2}.
\end{multline*}
Taking $T = c t_g^{\frac{3}{2} (1 - \delta)}$ yields the result.
\end{proof}

\section{Equidistribution of Geometric Invariants of Quadratic Fields}
\label{section6}

\subsection{Geometric Invariants of Quadratic Fields}\label{geominvsect}

Let $K = \Q(\sqrt{D})$ be a quadratic field of discriminant $D$. We denote by $h_K^+ \defeq \# \Cl_K^+$ the narrow class number of $K$ and $h_K \defeq \# \Cl_K$ the (wide) class number of $K$; note that $\Cl_K^+ = \Cl_K$, so that $h_K^+ = h_K$, except when $D > 1$ and $\OO_K^{\times}$ contains no elements of norm $-1$, in which case $h_K^+ = 2h_K$. Each narrow ideal class $A$ of $\Cl_K^+$ is associated to an $\SL_2(\Z)$-equivalence class of binary quadratic forms $Q(x,y) = ax^2 + bxy + cy^2$ of discriminant $D$.

Associated to equivalence classes of binary quadratic forms are geometric invariants: if $D < 0$, this is a Heegner point $z_A \in \Gamma \backslash \Hb$, while if $D > 0$, these are a closed geodesic $\CC_A \subset \Gamma \backslash \Hb$ and a hyperbolic orbifold $\Gamma_A \backslash \NN_A$ whose boundary is $\CC_A$. This last geometric invariant was introduced by Duke, Imamo\={g}lu, and T\'{o}th in \cite{DIT}.

\subsubsection{Heegner Points \texorpdfstring{$z_A$}{z\9040\220}}

Given a binary quadratic form $Q(x,y) = ax^2 + bxy + cy^2$ of discriminant $b^2 - 4ac = D < 0$, the point
\[z = \frac{-b + i \sqrt{-D}}{2a}\]
lies in $\Hb$. The equivalence class of binary quadratic forms containing $Q(x,y)$, and hence the corresponding ideal class $A \in \Cl_K$, thereby corresponds to a point $z = z_A$ in $\Gamma \backslash \Hb$, which we call a Heegner point.

\subsubsection{Closed Geodesics \texorpdfstring{$\CC_A$}{C\9040\220}}

Given a binary quadratic forms $Q(x,y) = ax^2 + bxy + cy^2$ of discriminant $b^2 - 4ac = D > 0$, the points
\[\frac{-b \pm \sqrt{D}}{2a}\]
determine the endpoints of a closed geodesic in $\Hb$. The equivalence class of binary quadratic forms containing $Q(x,y)$ thereby corresponds to a closed geodesic $\CC = \CC_A$ in $\Gamma \backslash \Hb$. The length
\[\ell(\CC_A) \defeq \int_{\CC_A} \, ds\]
of $\CC_A$, with $ds^2 = y^{-2} dx^2 + y^{-2} dy^2$, is equal to $2 \log \epsilon_K^+$, where $\epsilon_K^+ > 1$ is the smallest unit with positive norm in the ring of integers $\OO_K$ of $K$, so that $\epsilon_K^+ = (x + y \sqrt{D})/2$ with $(x,y) \in \R_+^2$ the fundamental solution to the Pell equation $x^2 - D y^2 = 4$. Note that $\epsilon_K^+$ is equal to $\epsilon_K$, the fundamental unit of $K$, if $\OO_K^{\times}$ contains no elements of norm $-1$, whereas $\epsilon_K^+ = \epsilon_K^2$ if $\OO_K^{\times}$ does contain elements of norm $-1$.

\subsubsection{Hyperbolic Orbifolds \texorpdfstring{$\Gamma_A \backslash \NN_A$}{\textGamma\9040\220 \textbackslash N\9040\220}}

Let $K = \Q(\sqrt{D})$ be a real quadratic field of discriminant $D > 1$. Associated to a narrow ideal class $A \in \Cl_K^+$ is an invariant $((n_1,\ldots,n_{\ell_A}))$, where $\ell_A$ is a positive integer and $n_1,\ldots,n_{\ell_A}$ are integers; this is the primitive cycle, unique up to cyclic permutations, occurring in the minus continued fraction expansion of each point $w \in K$ for which $1 > w > \sigma(w) > 0$ and $w\Z + \Z \in A$. We define the elements
\[S \defeq \pm \begin{pmatrix} 0 & 1 \\ -1 & 0 \end{pmatrix}, \qquad T \defeq \pm \begin{pmatrix} 1 & 1 \\ 0 & 1 \end{pmatrix}\]
of $\PSL_2(\Z)$, which generate $\PSL_2(\Z)$ as the free product of $S$ and $T$. For each $k \in \{1, \ldots, \ell_A\}$, define
\[S_k \defeq T^{n_1 + \cdots + n_k} S T^{-n_1 - \cdots - n_k}.\]
This is an elliptic element of order $2$ in $\PSL_2(\Z)$. We set
\[\Gamma_A \defeq \left\langle S_1, \cdots, S_{\ell_A}, T^{n_1 + \cdots + n_{\ell_A}} \right\rangle,\]
which is a thin subgroup of $\PSL_2(\Z)$. The Nielsen region $\NN_A$ of $\Gamma_A$ is the smallest nonempty $\PSL_2(\Z)$-invariant open convex subset of $\Hb$. Then $\Gamma_A \backslash \NN_A$ is a hyperbolic orbifold, which naturally projects onto $\Gamma \backslash \Hb$. The boundary of $\Gamma_A \backslash \NN_A$ is a simple closed geodesic whose image in $\Gamma \backslash \Hb$ is $\CC_A$, and the volume of $\Gamma_A \backslash \NN_A$ is $\pi \ell_A$.

\begin{remark}
In fact, $\Gamma_A$ depends on the choice of $w$. The resulting hyperbolic orbifold $\Gamma_A \backslash \NN_A$ ends up being only unique up to translation; however, the projection of $\Gamma_A \backslash \NN_A$ onto $\Gamma \backslash \Hb$ is independent of the choice of $w$.
\end{remark}

\subsection{Weyl Sums}

\subsubsection{Variances and Weyl Sums}

We define
\begin{align*}
& \Var\left(G_K(z_A);R\right) \defeq \int_{\Gamma \backslash \Hb} \left(\frac{\#\left\{A \in G_K : z_A \in B_R(w)\right\}}{\vol\left(B_R\right) \# G_K} - \frac{1}{\vol\left(\Gamma \backslash \Hb\right)}\right)^2 \, d\mu(w),	\\
& \Var\left(G_K(\CC_A);R\right) \defeq \int_{\Gamma \backslash \Hb} \left(\frac{\sum_{A \in G_K} \ell\left(\CC_A \cap B_R(w)\right)}{\vol\left(B_R\right) \sum_{A \in G_K} \ell\left(\CC_A\right)} - \frac{1}{\vol\left(\Gamma \backslash \Hb\right)}\right)^2 \, d\mu(w),	\\
& \Var\left(G_K(\Gamma_A \backslash \NN_A);R\right)	\\
& \hspace{2.605cm} \defeq \int_{\Gamma \backslash \Hb} \left(\frac{\sum_{A \in G_K} \vol\left(\Gamma_A \backslash \NN_A \cap B_R(w)\right)}{\vol\left(B_R\right) \sum_{A \in G_K} \vol\left(\Gamma_A \backslash \NN_A\right)} - \frac{1}{\vol\left(\Gamma \backslash \Hb\right)}\right)^2 \, d\mu(w).
\end{align*}

The proofs of \hyperref[quadalmosteverycondthm]{Theorems \ref*{quadalmosteverycondthm}} and \ref{quadalmosteveryuncondthm} follow via Chebyshev's inequality from the following two propositions.

\begin{proposition}\label{shrinkingquadcondprop}
Suppose that $R \asymp |D|^{-\delta}$. Assuming the generalised Lindel\"{o}f hypothesis, we have that as $D \to -\infty$ along fundamental discriminants,
\begin{align*}
\Var\left(G_K(z_A);R\right) & \ll_{\e} (-D)^{-\left(\frac{1}{4} - \delta\right) + \e} \quad \text{for $0 < \delta < 1/4$,}
\intertext{while as $D \to \infty$ along fundamental discriminants,}
\Var\left(G_K(\CC_A);R\right) & \ll_{\e} D^{-\left(\frac{1}{2} - \delta\right) + \e} \quad \text{for $0 < \delta < 1/2$.}
\end{align*}
\end{proposition}

\begin{proposition}\label{shrinkingquaduncondprop}
Suppose that $R \asymp |D|^{-\delta}$. Then as $D \to -\infty$ along odd fundamental discriminants,
\begin{align*}
\Var\left(G_K(z_A);R\right) & \ll_{\e} (-D)^{-\left(\frac{1}{12} - \delta\right) + \e} \quad \text{for $0 < \delta < 1/12$,}
\intertext{while as $D \to \infty$ along odd fundamental discriminants,}
\Var\left(G_K(\CC_A);R\right) & \ll_{\e} D^{-\left(\frac{1}{6} - \delta\right) + \e} \quad \text{for $0 < \delta < 1/6$,}	\\
\Var\left(G_K(\Gamma_A \backslash \NN_A);R\right) & \ll_{\e} D^{-\frac{1}{2} + \e} \quad \text{for all $\delta > 0$.}
\end{align*}
\end{proposition}

We begin by determining the spectral expansions of these variances. For $f \in \BB_0(\Gamma)$, we define the Weyl sums
\begin{align*}
W_{G_K(z_A),f} & \defeq \sum_{A \in G_K} f\left(z_A\right),	\\
W_{G_K(\CC_A),f} & \defeq \sum_{A \in G_K} \int_{\CC_A} f(z) \, ds,	\\
W_{G_K(\Gamma_A \backslash \NN_A),f} & \defeq \sum_{A \in G_K} \int_{\Gamma_A \backslash \NN_A} f(z) \, d\mu(z).
\end{align*}
We define $W_{G_K(z_A),\infty}(t)$, $W_{G_K(\CC_A),\infty}(t)$, and $W_{G_K(\Gamma_A \backslash \NN_A),\infty}(t)$ similarly with $f$ replaced by $E(\cdot,1/2 + it)$.

\begin{proposition}\label{VarGKprop}
We have that
\begin{align*}
\Var\left(G_K(z_A);R\right) & = \sum_{f \in \BB_0(\Gamma)} \left|h_R\left(t_f\right)\right|^2 \frac{\left|W_{G_K(z_A),f}\right|^2}{\left(\# G_K\right)^2}	\\
& \hspace{2cm} + \frac{1}{4\pi} \int_{-\infty}^{\infty} \left|h_R(t)\right|^2 \frac{\left|W_{G_K(z_A),\infty}(t)\right|^2}{\left(\# G_K\right)^2} \, dt,	\\
\Var\left(G_K(\CC_A);R\right) & = \sum_{f \in \BB_0(\Gamma)} \left|h_R\left(t_f\right)\right|^2 \frac{\left|W_{G_K(\CC_A),f}\right|^2}{\left(\sum_{A \in G_K} \ell\left(\CC_A\right)\right)^2}	\\
& \hspace{2cm} + \frac{1}{4\pi} \int_{-\infty}^{\infty} \left|h_R(t)\right|^2 \frac{\left|W_{G_K(\CC_A),\infty}(t)\right|^2}{\left(\sum_{A \in G_K} \ell\left(\CC_A\right)\right)^2} \, dt,	\\
\Var\left(G_K(\Gamma_A \backslash \NN_A);R\right) & = \sum_{f \in \BB_0(\Gamma)} \left|h_R\left(t_f\right)\right|^2 \frac{\left|W_{G_K(\Gamma_A \backslash \NN_A),f}\right|^2}{\left(\sum_{A \in G_K} \vol\left(\Gamma_A \backslash \NN_A\right)\right)^2}	\\
& \hspace{2cm} + \frac{1}{4\pi} \int_{-\infty}^{\infty} \left|h_R(t)\right|^2 \frac{\left|W_{G_K(\Gamma_A \backslash \NN_A),\infty}(t)\right|^2}{\left(\sum_{A \in G_K} \vol\left(\Gamma_A \backslash \NN_A\right)\right)^2} \, dt.
\end{align*}
\end{proposition}

\begin{proof}
This follows from the spectral expansion of $K_R$ and Parseval's identity.
\end{proof}

To bound these variances, we require upper bounds for the Weyl sums as well as lower bounds for $\# G_K$, $\sum_{A \in G_K} \ell\left(\CC_A\right)$, and $\sum_{A \in G_K} \vol\left(\Gamma_A \backslash \NN_A\right)$.

\begin{lemma}\label{GKboundlemma}
We have that
\begin{gather*}
(-D)^{\frac{1}{2} - \e} \ll_{\e} \# G_K \ll \sqrt{-D} \log(-D),	\\
D^{\frac{1}{2} - \e} \ll_{\e} \sum_{A \in G_K} \ell\left(\CC_A\right) \ll \sqrt{D} \log D,	\\
D^{\frac{1}{2} - \e} \ll_{\e} \sum_{A \in G_K} \vol\left(\Gamma_A \backslash \NN_A\right) \ll \sqrt{D} \log D.
\end{gather*}
\end{lemma}

\begin{proof}
We have that $\# G_K = 2^{1 - \omega(|D|)} h_K^+$ and $\ell(\CC_A) = 2 \log \epsilon_K^+$, while it is shown in \cite[Proposition 1]{DIT} that
\[\frac{\# G_K \log \epsilon_K^+}{\log D} \ll \sum_{A \in G_K} \vol\left(\Gamma_A \backslash \NN_A\right) \ll \# G_K \log \epsilon_K^+.\]
The class number formula states that
\[h_K^+ = \begin{dcases*}
\frac{\sqrt{D} L\left(1, \chi_D\right)}{\log \epsilon_K^+} & if $D > 0$,	\\
\frac{w_K \sqrt{-D} L\left(1, \chi_D\right)}{2\pi} & if $D < 0$,
\end{dcases*}\]
where
\[w_K \defeq \# \OO_{K,\tors}^{\times} = \begin{dcases*}
4 & if $D = -4$,	\\
6 & if $D = -3$,	\\
2 & otherwise.
\end{dcases*}\]
The result then follows from the Landau--Siegel theorem and the bound $L(1,\chi_D) \ll \log |D|$.
\end{proof}

\subsubsection{Genus Characters}

The character group $\widehat{\Gen}_K$ of $\Gen_K$ is the group of real characters of $\Cl_K^+$. These genus characters are indexed by unordered pairs of coprime fundamental discriminants $d_1,d_2 \in \Z$ satisfying $d_1 d_2 = D$. To each pair $d_1,d_2$, we let $\chi = \chi_{d_1,d_2}$ denote the genus character corresponding to $d_1,d_2$: this is a real character of the narrow class group $\Cl_K^+$ that extends multiplicatively to all nonzero fractional ideals via
\[\chi(\pp) \defeq \begin{dcases*}
\chi_{d_1}(N(\pp)) & if $(N(\pp), d_1) = 1$,	\\
\chi_{d_2}(N(\pp)) & if $(N(\pp), d_2) = 1$,
\end{dcases*}\]
for any prime ideals $\pp \nmid \dd_K$, where $\chi_{d_1}, \chi_{d_2}$ are the primitive real Dirichlet characters modulo $d_1,d_2$ respectively. In particular, $\chi$ is a quadratic character unless either $d_1$ or $d_2$ is $1$, in which case it is the trivial character.

\begin{lemma}
For any narrow ideal classes $A_1,A_2 \in \Cl_K^+$, we have that
\[\frac{1}{2^{\omega(|D|) - 1}} \sum_{\chi \in \widehat{\Gen}_K} \chi(A_1 A_2) = \begin{dcases*}
1 & if $A_2 \in A_1 (\Cl_K^+)^2$,	\\
0 & otherwise.
\end{dcases*}\]
\end{lemma}

\begin{proof}
This is character orthogonality for finite abelian groups.
\end{proof}

We abuse notation and write $G_K$ for an element in the coset of $\Cl_K^+$ corresponding to the genus $G_K$. This allows us to write
\begin{align*}
W_{G_K(z_A),f} & = \frac{1}{2^{\omega(-D) - 1}} \sum_{\chi \in \widehat{\Gen}_K} \chi(G_K) \sum_{A \in \Cl_K} \chi(A) f(z_A),	\\
W_{G_K(\CC_A),f} & = \frac{1}{2^{\omega(D) - 1}} \sum_{\chi \in \widehat{\Gen}_K} \chi(G_K) \sum_{A \in \Cl_K^+} \chi(A) \int_{\CC_A} f(z) \, ds,	\\
W_{G_K(\Gamma_A \backslash \NN_A),f} & = \frac{1}{2^{\omega(D) - 1}} \sum_{\chi \in \widehat{\Gen}_K} \chi(G_K) \sum_{A \in \Cl_K^+} \chi(A) \int_{\Gamma_A \backslash \NN_A} f(z) \, d\mu(z),
\end{align*}
and analogous identities for $W_{G_K(z_A),\infty}(t)$, $W_{G_K(\CC_A),\infty}(t)$, and $W_{G_K(\Gamma_A \backslash \NN_A),\infty}(t)$. This has the advantage that we are able to show in each case that the square of the sum over $A \in \Cl_K^+$ is essentially equal to a product of $L$-functions.

\subsubsection{Maa\ss{} Form Weyl Sums}

\begin{lemma}\label{MaassWeyllemma}
We have that
\begin{align*}
\left|W_{G_K(z_A),f}\right|^2 & \ll \sqrt{-D} \sum_{\chi \in \widehat{\Gen}_K} \frac{L\left(\frac{1}{2}, f \otimes \chi_{d_1}\right) L\left(\frac{1}{2}, f \otimes \chi_{d_2}\right)}{L\left(1, \sym^2 f\right)},	\\
%\displaybreak
\left|W_{G_K(\CC_A),f}\right|^2 & \ll \sqrt{\frac{D}{\frac{1}{4} + t_f^2}} \sum_{\substack{\chi \in \widehat{\Gen}_K \\ d_1, d_2 > 0}} \frac{L\left(\frac{1}{2}, f \otimes \chi_{d_1}\right) L\left(\frac{1}{2}, f \otimes \chi_{d_2}\right)}{L\left(1, \sym^2 f\right)},	\\
\left|W_{G_K(\Gamma_A \backslash \NN_A),f}\right|^2 & \ll \sqrt{\frac{D}{\left(\frac{1}{4} + t_f^2\right)^{3}}} \sum_{\substack{\chi \in \widehat{\Gen}_K \\ d_1, d_2 < 0}} \frac{L\left(\frac{1}{2}, f \otimes \chi_{d_1}\right) L\left(\frac{1}{2}, f \otimes \chi_{d_2}\right)}{L\left(1, \sym^2 f\right)}.
\end{align*}
\end{lemma}

\begin{proof}
For $\chi = \chi_{d_1,d_2}$ and even $f \in \BB_0(\Gamma)$, it is shown in \cite[Theorem 4 and Equation (5.17)]{DIT} that the quantity
\begin{equation}\label{toricintegral}
\begin{dcases*}
\left|\sum_{A \in \Cl_K} \chi(A) \frac{4 \sqrt{\pi}}{w_K} f\left(z_A\right)\right|^2 & if $d_1 d_2 < 0$,	\\
\left|\sum_{A \in \Cl_K^+} \chi(A) \int_{\CC_A} f(z) \, ds\right|^2 & if $d_1,d_2 > 0$,	\\
\left|\sum_{A \in \Cl_K^+} \chi(A) \frac{\frac{1}{4} + t_f^2}{2} \int_{\Gamma_A \backslash \NN_A} f(z) \, d\mu(z)\right|^2 & if $d_1,d_2 < 0$
\end{dcases*}
\end{equation}
is equal to
\[\frac{1}{2} \frac{\Lambda\left(\frac{1}{2}, f \otimes \chi_{d_1}\right) \Lambda\left(\frac{1}{2}, f \otimes \chi_{d_2}\right)}{\Lambda\left(1, \sym^2 f\right)}.\]
Here we recall the definition \eqref{Lambda(s,pi)} of the completed $L$-function, and in particular that this includes the conductor. This identity also holds when $f$ is odd, because in this case $L(1/2, f \otimes \chi_d) = 0$. Finally, it is also shown that
\[\sum_{A \in \Cl_K^+} \chi(A) \int_{\Gamma_A \backslash \NN_A} f(z) \, d\mu(z)\]
vanishes if $d_1,d_2 > 0$, and similarly
\[\sum_{A \in \Cl_K^+} \chi(A) \int_{\CC_A} f(z) \, ds\]
vanishes if $d_1,d_2 < 0$. The result then follows from the Cauchy-Schwarz inequality and Stirling's approximation.
\end{proof}

\begin{remark}
The terms \eqref{toricintegral} can be viewed as toric integrals in the sense of \cite[Section 2.2.1]{MV06}, and these can be generalised to involve Hecke Gr\"{o}\ss{}encharaktere $\chi$ of $K$ that are not necessarily genus characters. The resulting toric integral in this generalised setting will essentially be equal to the completed Rankin--Selberg $L$-function $\Lambda(1/2, f \otimes g_{\chi})$, where $g_{\chi}$ denotes the automorphic induction of the Hecke Gr\"{o}\ss{}encharakter $\chi$ to a Maa\ss{} newform $g_{\chi}$. When $\chi$ is a genus character $\chi_{d_1,d_2}$, this Rankin--Selberg $L$-function factorises as $\Lambda(1/2, f \otimes \chi_{d_1}) \Lambda(1/2, f \otimes \chi_{d_2})$, while the case of $\chi$ being an ideal class character of an imaginary quadratic field $K$ and its applications towards equidistribution of Heegner points in conjugates of $\Gamma \backslash \Hb$ in $\Gamma_0(q) \backslash \Hb$ is investigated in \cite{LMY}.
\end{remark}

\subsubsection{Eisenstein Series Weyl Sums}

\begin{lemma}\label{EisensteinWeyllemma}
We have that
\begin{align*}
\left|W_{G_K(z_A),\infty}(t)\right|^2 & \ll \sqrt{-D} \sum_{\chi \in \widehat{\Gen}_K} \left|\frac{L\left(\frac{1}{2} + it, \chi_{d_1}\right) L\left(\frac{1}{2} + it, \chi_{d_2}\right)}{\zeta(1 + 2it)}\right|^2,	\\
%\displaybreak
\left|W_{G_K(\CC_A),\infty}(t)\right|^2 & \ll \sqrt{\frac{D}{\frac{1}{4} + t^2}} \sum_{\substack{\chi \in \widehat{\Gen}_K \\ d_1, d_2 > 0}} \left|\frac{L\left(\frac{1}{2} + it, \chi_{d_1}\right) L\left(\frac{1}{2} + it, \chi_{d_2}\right)}{\zeta(1 + 2it)}\right|^2,	\\
\left|W_{G_K(\Gamma_A \backslash \NN_A),\infty}(t)\right|^2 & \ll \sqrt{\frac{D}{\left(\frac{1}{4} + t^2\right)^3}} \sum_{\substack{\chi \in \widehat{\Gen}_K \\ d_1, d_2 < 0}} \left|\frac{L\left(\frac{1}{2} + it, \chi_{d_1}\right) L\left(\frac{1}{2} + it, \chi_{d_2}\right)}{\zeta(1 + 2it)}\right|^2.
\end{align*}
\end{lemma}

\begin{proof}
This follows from \cite[Theorem 3]{DIT}, akin to the proof of \hyperref[MaassWeyllemma]{Lemma \ref*{MaassWeyllemma}}.
\end{proof}

\subsection{Bounds for the Variances}

\begin{proof}[Proof of \texorpdfstring{\hyperref[shrinkingquadcondprop]{Proposition \ref*{shrinkingquadcondprop}}}{Proposition \ref{shrinkingquadcondprop}}]
For $R \asymp (-D)^{-\delta}$, $\Var\left(G_K(z_A);R\right)$ is bounded by a constant multiple dependent on $\e$ of
\begin{multline*}
(-D)^{-\frac{1}{2} + \e} \sum_{\chi \in \widehat{\Gen}_K} \sum_{0 < t_f < 2(-D)^{\delta}} \frac{L\left(\frac{1}{2}, f \otimes \chi_{d_1}\right) L\left(\frac{1}{2}, f \otimes \chi_{d_2}\right)}{L\left(1, \sym^2 f\right)}	\\
+ (-D)^{-\frac{1}{2} + 3\delta + \e} \sum_{\chi \in \widehat{\Gen}_K} \sum_{t_f \geq 2(-D)^{\delta}} \frac{L\left(\frac{1}{2}, f \otimes \chi_{d_1}\right) L\left(\frac{1}{2}, f \otimes \chi_{d_2}\right)}{t_f^3 L\left(1, \sym^2 f\right)}	\\
+ (-D)^{-\frac{1}{2} + \e} \sum_{\chi \in \widehat{\Gen}_K} \int_{0}^{2(-D)^{\delta}} \frac{\left|L\left(\frac{1}{2} + it, \chi_{d_1}\right)\right|^2 \left|L\left(\frac{1}{2} + it, \chi_{d_2}\right)\right|^2}{\left|\zeta(1 + 2it)\right|^2} \, dt	\\
+ (-D)^{-\frac{1}{2} + 3\delta + \e} \sum_{\chi \in \widehat{\Gen}_K} \int_{2(-D)^{\delta}}^{\infty} \frac{\left|L\left(\frac{1}{2} + it, \chi_{d_1}\right)\right|^2 \left|L\left(\frac{1}{2} + it, \chi_{d_2}\right)\right|^2}{t^3 \left|\zeta(1 + 2it)\right|^2} \, dt
\end{multline*}
via \hyperref[VarGKprop]{Proposition \ref*{VarGKprop}} and \hyperref[GKboundlemma]{Lemmata \ref*{GKboundlemma}}, \ref{MaassWeyllemma}, and \ref{EisensteinWeyllemma}; an analogous bound also holds for $\Var(G_K(\CC_A);R)$. Making use of the generalised Lindel\"{o}f hypothesis in each expression and using the Weyl law yields \hyperref[shrinkingquadcondprop]{Proposition \ref*{shrinkingquadcondprop}}.
\end{proof}

For unconditional results, we make use of the following bounds.

\begin{lemma}[{\cite[Theorem]{Ivi}}]\label{Dtrivsubconvexsumlemma}
For $T \gg 1$,
\begin{align*}
\sum_{T \leq t_f \leq T + 1} \frac{L\left(\frac{1}{2}, f\right)^3}{L(1,\sym^2 f)} & \ll_{\e} T^{1 + \e},	\\
\int_{T}^{T + 1} \frac{\left|\zeta\left(\frac{1}{2} + it\right)\right|^6}{\left|\zeta(1 + 2it)\right|^2} \, dt & \ll_{\e} T^{1 + \e}.
\end{align*}
\end{lemma}

\begin{lemma}[{\cite[Theorem 1.1]{You17a}}]\label{Dnontrivsubconvexsumlemma}
For odd fundamental discriminants $D \neq 1$ and $T \gg 1$,
\begin{align*}
\sum_{T \leq t_f \leq T + 1} \frac{L\left(\frac{1}{2}, f \otimes \chi_D\right)^3}{L(1,\sym^2 f)} & \ll_{\e} (|D| T)^{1 + \e},	\\
\int_{T}^{T + 1} \frac{\left|L\left(\frac{1}{2} + it, \chi_D\right)\right|^6 \, dt}{\left|\zeta(1 + 2it)\right|^2} & \ll_{\e} (|D| T)^{1 + \e}.
\end{align*}
\end{lemma}

\begin{proof}[Proof of \texorpdfstring{\hyperref[shrinkingquaduncondprop]{Proposition \ref*{shrinkingquaduncondprop}}}{Proposition \ref{shrinkingquaduncondprop}}]
We bound the variance by breaking up into ranges as in the proof of \hyperref[shrinkingquadcondprop]{Proposition \ref{shrinkingquadcondprop}}. Instead of applying the generalised Lindel\"{o}f hypothesis, we use the generalised H\"{o}lder inequality with exponents $(3,3,3)$. Via the bounds in \hyperref[Dtrivsubconvexsumlemma]{Lemmata \ref*{Dtrivsubconvexsumlemma}} and \ref{Dnontrivsubconvexsumlemma}, together with the Weyl law, we obtain the result.
\end{proof}

\subsection{Representations of Integers by Indefinite Ternary Quadratic Forms}

We briefly describe how the results in this section can be interpreted in terms of indefinite ternary quadratic forms. For simplicity, we only discuss the case of negative discriminant and summing over all genera; for positive discriminant, a detailed presentation can be found in \cite[Section 2]{ELMV}.

Consider the indefinite ternary quadratic form
\[Q(a,b,c) = b^2 - 4ac.\]
We are interested in the level sets
\[V_{Q,D}(\Z) \defeq \left\{(a,b,c) \in \Z^3 : b^2 - 4ac = D\right\},\]
where $D < 0$ is a fundamental discriminant; these sets parametrise the different ways that the integer $D$ can be represented by the ternary quadratic form $Q$. The normalised level set $\Gs_D \defeq (-D)^{-1/2} V_{Q,D}(\Z)$ lies inside the two-sheeted hyperboloid
\[V_{Q,-1}(\R) \defeq \left\{(a,b,c) \in \R^3 : b^2 - 4ac = -1\right\}.\]

It is natural to ask whether the normalised level sets $\Gs_D$ cover $V_{Q,-1}(\R)$ randomly as $D$ tends to $-\infty$ along fundamental discriminants. Each level set $V_{Q,D}(\Z)$ is countably infinite, and $V_{Q,-1}(\R)$ is isomorphic to $\C \setminus \R$, which is not of finite volume, so one cannot immediately rephrase this random covering as equidistribution.

On the other hand, the group
\[\SO_Q(\Z) \defeq \left\{A \in \SL_3(\Z) : Q(Ax) = Q(x) \text{ for all $x = (a,b,c) \in \Z^3$}\right\}\]
acts transitively on $V_{Q,D}(\Z)$, and the quotient space $\SO_Q(\Z) \backslash \Gs_D$ is finite for all fundamental discriminants $D$, with cardinality equal to $h_K$. Moreover, $\SO_Q(\Z)$ is a discrete subgroup of $\SO_Q(\R)$ of finite covolume, and $V_{Q,-1}(\R) \cong \SO_Q(\R) / K$ with $K$ equal to the maximal compact subgroup of $\SO_Q(\R)$, and so the space $\SO_Q(\Z) \backslash V_{Q,-1}(\R)$ is of finite volume.

Thus to ask whether the normalised level sets $\Gs_D$ randomly cover $V_{Q,-1}(\R)$ can be rephrased as asking whether the finite sets $\SO_Q(\Z) \backslash \Gs_D$ equidistribute in the finite volume space $\SO_Q(\Z) \backslash V_{Q,-1}(\R)$. This has a positive answer by naturally realising this result in terms of the equidistribution of Heegner points on $\Gamma \backslash \Hb$, as proved by Duke \cite[Theorem 1]{Duk}. Indeed, the fact that $Q$ is indefinite implies that $\SO_Q$ is isomorphic to the split special orthogonal group $\SO_{1,2}$, and we have the accidental isomorphism $\SO_{1,2} \cong \PGL_2$, while $K \cong \SO_2(\R)$. From this, we see that $\SO_Q(\Z) \backslash V_{Q,-1}(\R) \cong \PGL_2(\Z) \backslash \PGL_2(\R) / \SO_2(\R) \cong \Gamma \backslash \Hb$, while $\SO_Q(\Z) \backslash \Gs_D$ is naturally identified with the set of Heegner points $\{z_A \in \Gamma \backslash \Hb : A \in \Cl_K\}$.

With this reinterpretation in mind, we now see that \hyperref[shrinkingquadcondprop]{Proposition \ref*{shrinkingquadcondprop}} implies that under the assumption of the generalised Lindel\"{o}f hypothesis, almost every shrinking ball of radius $R \asymp (-D)^{-\delta}$ with $0 < \delta < 1/4$ in $\SO_Q(\Z) \backslash V_{Q,-1}(\R)$ contains a normalised equivalence class of points $(a,b,c) \in \Z^3$ that represent the integer $D$ by the indefinite ternary quadratic form $Q(a,b,c) = b^2 - 4ac$. This complements \cite[Theorem 1.8]{BRS}, where the analogous result is proved for the definite ternary quadratic form $Q(a,b,c) = a^2 + b^2 + c^2$.

\subsection*{Acknowledgements}

The author thanks Peter Sarnak for suggesting this problem and many helpful discussions on this topic, as well as Matt Young for useful feedback.
%Djordje for random wave conjecture, referee for helpful comments.


\begin{thebibliography}{ELMV12}

\bibitem[Ber77]{Ber} M.~V.~Berry, ``Regular and Irregular Semiclassical Wavefunctions'', \textit{Journal of Physics A: Mathematical and General} \textbf{10}:12 (1977), 2083--2091. \textsc{doi}:\allowbreak\href{https://doi.org/10.1088/0305-4470/10/12/016}{10.1088/0305-4470/10/12/016}

\bibitem[BM15]{BM} Valentin Blomer and Djordje Mili\'{c}evi\'{c}, ``The Second Moment of Twisted Modular $L$-Functions'', \textit{Geometric and Functional Analysis} \textbf{25}:2 (2015), 453--516. \textsc{doi}:\allowbreak\href{https://doi.org/10.1007/s00039-015-0318-7}{10.1007/s00039-015-0318-7}

\bibitem[Bou17]{Bou} Jean Bourgain, ``Decoupling, Exponential Sums and the Riemann Zeta Function'', \textit{Journal of the American Mathematical Society} \textbf{30}:1 (2017), 205--224. \textsc{doi}:\allowbreak\href{https://doi.org/10.1090/jams/860}{10.1090/jams/860}

\bibitem[BRS16]{BRS} Jean Bourgain, Ze\'{e}v Rudnick, and Peter Sarnak, ``Spatial Statistics for Lattice Points on the Sphere I: Individual Results'', \textit{Bulletin of the Iranian Mathematical Society} \textbf{43}:4 (2017), 361--38. \url{http://bims.iranjournals.ir/article_1169.html}

\bibitem[BK17a]{BK17a} Jack Buttcane and Rizwanur Khan, ``A Mean Value of Triple Product $L$-Functions'', \textit{Mathematische Zeitschrift} \textbf{285}:1 (2017), 565--591. \textsc{doi}:\allowbreak\href{https://doi.org/10.1007/s00209-016-1721-y}{10.1007/s00209-016-1721-y}

\bibitem[BK17b]{BK17b} Jack Buttcane and Rizwanur Khan, ``On the Fourth Moment of Hecke Maass Forms and the Random Wave Conjecture'', \textit{Compositio Mathematica} \textbf{153}:7 (2017), 1479--1511. \textsc{doi}:\allowbreak\href{https://doi.org/10.1112/S0010437X17007199}{10.1112/S0010437X17007199}

\bibitem[Cha96]{Cha96} Fernando Chamizo, ``Some Applications of Large Sieve in Riemann Surfaces'', \textit{Acta Arithmetica} \textbf{77}:4 (1996), 315--337. \textsc{doi}:\allowbreak\href{https://doi.org/10.4064/aa-77-4-315-337}{10.4064/aa-77-4-315-337}

\bibitem[Che04]{Che} Thyagaraju Chelluri, \textit{Equidistribution of the Roots of Quadratic Congruences}, Ph.D.~Thesis, Rutgers The State University of New Jersey, New Brunswick, 2004.

\bibitem[DK18]{DK18} Goran Djankovi\'{c} and Rizwanur Khan, ``A Conjecture for the Regularized Fourth Moment of Eisenstein Series'', \textit{Journal of Number Theory} \textbf{182} (2018), 236--257. \textsc{doi}:\allowbreak\href{https://doi.org/10.1016/j.jnt.2017.06.012}{10.1016/j.jnt.2017.06.012}

\bibitem[Duk88]{Duk} W.~Duke, ``Hyperbolic Distribution Problems and Half-Integral Weight Maass Forms'', \textit{Inventiones Mathematicae} \textbf{92}:1 (1988), 73--90. \textsc{doi}:\allowbreak\href{https://doi.org/10.1007/BF01393993}{10.1007/BF01393993}

\bibitem[DIT16]{DIT} W.~Duke, \"{O}.~Imamo\={g}lu, and \'{A}.~T\'{o}th, ``Geometric Invariants for Real Quadratic Fields'', \textit{Annals of Mathematics} \textbf{184}:3 (2016), 949--990. \textsc{doi}:\allowbreak\href{https://doi.org/10.4007/annals.2016.184.3.8}{10.4007/annals.2016.184.3.8}

\bibitem[ELMV12]{ELMV} Manfred Einsiedler, Elon Lindenstrauss, Philippe Michel, and Akshay Venkatesh, ``The Distribution of Closed Geodesics on the Modular Surface, and Duke's Theorem'', \textit{L'Enseignement Math\'{e}matique} \textbf{58} (2012), 249--313. \textsc{doi}:\allowbreak\href{https://doi.org/10.4171/LEM/58-3-2}{10.4171/LEM/58-3-2}

\bibitem[EMV13]{EMV} Jordan S.~Ellenberg, Philippe Michel, and Akshay Venkatesh, ``Linnik's Ergodic Method and the Distribution of Integer Points on Spheres'' in \textit{Automorphic Representations and $L$-Functions. Proceedings of the International Colloquium, Mumbai 2012}, editors D.~Prasad, C.~S.~Rajan, A.~Sankaranarayanan, and J.~Sengupta, Hindustan Book Agency, New Delhi, 2013, 119--185.

\bibitem[GR07]{GR} I.~S.~Gradshteyn and I.~M.~Ryzhik, \textit{Table of Integrals, Series, and Products, Seventh Edition}, editors Alan Jeffrey and Daniel Zwillinger, Academic Press, Burlington, 2007.

\bibitem[GW17]{GW} Andrew Granville and Igor Wigman, ``Planck-Scale Mass Equidistribution of Toral Laplace Eigenfunctions'', \textit{Communications in Mathematical Physics} \textbf{355}:2 (2017), 767--802. \textsc{doi}:\allowbreak\href{https://doi.org/10.1007/s00220-017-2953-3}{10.1007/s00220-017-2953-3}

\bibitem[Han15]{Han15} Xiaolong Han, ``Small Scale Quantum Ergodicity in Negatively Curved Manifolds'', \textit{Nonlinearity} \textbf{28}:9 (2015), 3263--3288. \textsc{doi}:\allowbreak\href{https://doi.org/10.1088/0951-7715/28/9/3263}{10.1088/0951-7715/28/9/3263}

\bibitem[Han17]{Han17} Xiaolong Han, ``Small Scale Quantum Ergodicity of Random Eigenbases'', \textit{Communications in Mathematical Physics} \textbf{349}:1 (2017), 425--440. \textsc{doi}:\allowbreak\href{https://doi.org/10.1007/s00220-016-2597-8}{10.1007/s00220-016-2597-8}

\bibitem[HT16]{HT} Xiaolong Han and Melissa Tacy, ``Equidistribution of Random Waves on Small Balls'', preprint (2016), 13 pages. arXiv:\allowbreak\href{https://arxiv.org/abs/1611.05983}{1611.05983 [math.SP]}

\bibitem[Hej99]{Hej} Dennis A.~Hejhal, ``On Eigenfunctions of the Laplacian for Hecke Triangle Groups'',  in \textit{Emerging Applications of Number Theory}, editors Dennis A.~Hejhal, Joel Friedman, Martin C.~Gutzwiller, and Andrew M.~Odlyzko, The IMA Volumes in Mathematics and Its Applications \textbf{109}, Springer--Verlag, New York, 1999, 291--315. \textsc{doi}:\allowbreak\href{https://doi.org/10.1007/978-1-4612-1544-8_11}{10.1007/978-1-4612-1544-8\_11}

\bibitem[HejRa92]{HejRa} Dennis A.~Hejhal and Barry N.~Rackner, ``On the Topography of Maass Waveforms for $\PSL(2,\Z)$'', \textit{Experimental Mathematics} \textbf{1}:4 (1992), 275--305. \textsc{doi}:\allowbreak\href{https://doi.org/10.1080/10586458.1992.10504562}{10.1080/10586458.1992.10504562}

\bibitem[HeSt01]{HeSt} Dennis A.~Hejhal and Andreas Str\"{o}mbergsson, ``On Quantum Chaos and Maass Waveforms of CM-Type'', \textit{Foundations of Physics} \textbf{31}:3 (2001), 519--533. \textsc{doi}:\allowbreak\href{https://doi.org/10.1023/A:1017521729782}{10.1023/A:1017521729782}

\bibitem[HezRi16]{HezRi16} Hamid Hezari and Gabriel Rivi\`{e}re, ``$L^p$ Norms, Nodal Sets, and Quantum Ergodicity'', \textit{Advances in Mathematics} \textbf{290} (2016), 938--966. \textsc{doi}:\allowbreak\href{https://doi.org/10.1016/j.aim.2015.10.027}{10.1016/j.aim.2015.10.027}

\bibitem[HezRi17]{HezRi17} Hamid Hezari and Gabriel Rivi\`{e}re, ``Quantitative Equidistribution Properties of Toral Eigenfunctions'', \textit{Journal of Spectral Theory} \textbf{7}:2 (2017), 471--485. \textsc{doi}:\allowbreak\href{https://doi.org/10.4171/JST/169}{10.4171/JST/169}

\bibitem[HL94]{HL} Jeffrey Hoffstein and Paul Lockhart, ``Coefficients of Maass Forms and the Siegel Zero'', \textit{Annals of Mathematics} \textbf{140}:1 (1994), 161--176. \textsc{doi}:\allowbreak\href{https://doi.org/10.2307/2118543}{10.2307/2118543}

\bibitem[Hu17]{Hu} Yueke Hu, ``Triple Product Formula and Mass Equidistribution on Modular Curves of Level $N$'', to appear in \textit{International Mathematics Research Notices} (2017), 45 pages. \textsc{doi}:\allowbreak\href{https://doi.org/10.1093/imrn/rnw322}{10.1093/imrn/rnw322}

\bibitem[Ich08]{Ich} Atsushi Ichino, ``Trilinear Forms and the Central Values of Triple Product $L$-Functions'', \textit{Duke Mathematical Journal} \textbf{145}:2 (2008), 281--307. \textsc{doi}:\allowbreak\href{https://doi.org/10.1215/00127094-2008-052}{10.1215/00127094-2008-052} 

\bibitem[Ivi01]{Ivi} Aleksandar Ivi\'{c}, ``On Sums of Hecke Series in Short Intervals'', \textit{Journal de Th\'{e}orie des Nombres de Bordeaux} \textbf{13}:2 (2001), 453--468. \textsc{doi}:\allowbreak\href{https://doi.org/10.5802/jtnb.333}{10.5802/jtnb.333}

\bibitem[Iwa02]{Iwa} Henryk Iwaniec, \textit{Spectral Methods of Automorphic Forms, Second Edition}, Graduate Studies in Mathematics \textbf{53}, American Mathematical Society, Providence, 2002. \textsc{doi}:\allowbreak\href{https://doi.org/10.1090/gsm/053}{10.1090/gsm/053}

\bibitem[IK04]{IK} Henryk Iwaniec and Emmanuel Kowalski, \textit{Analytic Number Theory}, American Mathematical Society Colloquium Publications \textbf{53}, American Mathematical Society, Providence, 2004. \textsc{doi}:\allowbreak\href{https://doi.org/10.1090/coll/053}{10.1090/coll/053}

\bibitem[Jak94]{Jak} Dmitry Jakobson, ``Quantum Unique Ergodicity for Eisenstein Series on $\PSL_2(\Z) \backslash \PSL_2(\R)$'', \textit{Annales de l'Institut Fourier} \textbf{44}:5 (1994), 1477--1504. \textsc{doi}:\allowbreak\href{https://doi.org/10.5802/aif.1442}{10.5802/aif.1442}

\bibitem[Jut01]{Jut01} Matti Jutila, ``The Fourth Moment of Central Values of Hecke Series'', in \textit{Number Theory: Proceedings of the Turku Symposium on Number Theory in Memory of Kustaa Inkeri}, editors Matti Jutila and Tauno Mets\"{a}nkyl\"{a}, Walter de Gruyter, Berlin, 2001, 167--177.

\bibitem[Jut04]{Jut04} M.~Jutila, ``The Spectral Mean Square of Hecke $L$-Functions on the Critical Line'', \textit{Publications de l'Institut Math\'{e}matique, Nouvelle s\'{e}rie} \textbf{76}:90 (2004), 41--55. \textsc{doi}:\allowbreak\href{https://doi.org/10.2298/PIM0476041J}{10.2298/PIM0476041J}

\bibitem[JM05]{JM} Matti Jutila and Yoichi Motohashi, ``Uniform Bound for Hecke $L$-Functions'', \textit{Acta Mathematica} \textbf{195}:1 (2005), 61--115. \textsc{doi}:\allowbreak\href{https://doi.org/10.1007/BF02588051}{10.1007/BF02588051}

\bibitem[LMR15]{LMR} Stephen Lester, Kaisa Matom\"{a}ki, and Maksym Radziwi\l{}\l{}, ``Small Scale Distribution of Zeros and Mass of Modular Forms'', to appear in \textit{Journal of the European Mathematical Society} (2018), 31 pages. arXiv:\allowbreak\href{https://arxiv.org/abs/1501.01292}{1501.01292 [math.NT]}

\bibitem[LR17]{LR} Stephen Lester and Ze\'{e}v Rudnick, ``Small Scale Equidistribution of Eigenfunctions on the Torus'', \textit{Communications in Mathematical Physics} \textbf{350}:1 (2017), 279--300. \textsc{doi}:\allowbreak\href{https://doi.org/10.1007/s00220-016-2734-4}{10.1007/s00220-016-2734-4}

\bibitem[Lin06]{Lin} Elon Lindenstrauss, ``Invariant Measures and Arithmetic Quantum Unique Ergodicity'', \textit{Annals of Mathematics} \textbf{163}:1 (2006), 165--219. \textsc{doi}:\allowbreak\href{https://doi.org/10.4007/annals.2006.163.165}{10.4007/annals.2006.163.165}

\bibitem[LMY13]{LMY} Sheng-Chi Liu, Riad Masri, and Matthew P.~Young, ``Subconvexity and Equidistribution of Heegner Points in the Level Aspect'', \textit{Compositio Mathematica} \textbf{149}:7 (2013), 1150--1174. \textsc{doi}:\allowbreak\href{https://doi.org/10.1112/S0010437X13007033}{10.1112/S0010437X13007033}

\bibitem[Luo14]{Luo} Wenzhi Luo, ``$L^4$-Norms of the Dihedral Maass Forms'', \textit{International Mathematics Research Notices} \textbf{2014}:8 (2014), 2294--2304. \textsc{doi}:\allowbreak\href{https://doi.org/10.1093/imrn/rns298}{10.1093/imrn/rns298}

\bibitem[LS95]{LS} Wenzhi Luo and Peter Sarnak, ``Quantum Ergodicity of Eigenfunctions on $\PSL_2(\Z) \backslash \Hb^2$'', \textit{Publications Math\'{e}matiques de l'Institut des Hautes \'{E}tudes Scientifiques} \textbf{81}:1 (1995), 207--237. \textsc{doi}:\allowbreak\href{https://doi.org/10.1007/BF02699377}{10.1007/BF02699377}

\bibitem[MV06]{MV06} Philippe Michel and Akshay Venkatesh, ``Equidistribution, $L$-Functions and Ergodic Theory: On Some Problems of Yu.~Linnik'', in \textit{Proceedings of the International Congress of Mathematicians, Madrid 2006} \textbf{II}, editors Marta Sanz-Sol\'{e}, Javier Soria, Juan Luis Varona, and Joan Verdera, European Mathematical Society, Z\"{u}rich, 2006, 421--457. \href{http://www.icm2006.org/proceedings/Vol_II/contents/ICM_Vol_2_19.pdf}{http://www.icm2006.org/proceedings/Vol\_II/contents/ICM\_Vol\_2\_19.pdf}

\bibitem[Mil10]{Mil} Djordje Mili\'{c}evi\'{c}, ``Large Values of Eigenfunctions on Arithmetic Hyperbolic Surfaces'', \textit{Duke Mathematical Journal} \textbf{155}:2 (2010), 365--401. \textsc{doi}:\allowbreak\href{https://doi.org/10.1215/00127094-2010-058}{10.1215/00127094-2010-058}

\bibitem[NPS14]{NPS} Paul D.~Nelson, Ameya Pitale, and Abhishek Saha, ``Bounds for Rankin--Selberg Integrals and Quantum Unique Ergodicity for Powerful Levels'', \textit{Journal of the American Mathematical Society} \textbf{27}:1 (2014), 147--191. \textsc{doi}:\allowbreak\href{https://doi.org/10.1090/S0894-0347-2013-00779-1}{10.1090/S0894-0347-2013-00779-1}

\bibitem[Sar03]{Sar} Peter Sarnak, ``Spectra of Hyperbolic Surfaces'', \textit{Bulletin of the American Mathematical Society} \textbf{40}:4 (2003), 441--478. \textsc{doi}:\allowbreak\href{https://doi.org/10.1090/S0273-0979-03-00991-1}{10.1090/S0273-0979-03-00991-1}

\bibitem[Sou10]{Sou} Kannan Soundararajan, ``Quantum Unique Ergodicity for $\SL_2(\Z) \backslash \Hb$'', \textit{Annals of Mathematics} \textbf{172}:2 (2010), 1529--1538. \textsc{doi}:\allowbreak\href{https://doi.org/10.4007/annals.2010.172.1529}{10.4007/annals.2010.172.1529}

\bibitem[Spi03]{Spi} Florin Spinu, \textit{The $L^4$ Norm of the Eisenstein Series}, Ph.D.~Thesis, Princeton University, 2003. \href{http://www.math.jhu.edu/~fspinu/math/thesis.pdf}{http://www.math.jhu.edu/\textasciitilde{}fspinu/math/thesis.pdf}

\bibitem[Wat08]{Wat} Thomas C.~Watson, \textit{Rankin Triple Products and Quantum Chaos}, Ph.D.~Thesis, Princeton University, 2002 (revised 2008). arXiv:\allowbreak\href{https://arxiv.org/abs/0810.0425}{0810.0425 [math.NT]}

\bibitem[You16]{You16} Matthew P.~Young, ``The Quantum Unique Ergodicity Conjecture for Thin Sets'', \textit{Advances in Mathematics} \textbf{286} (2016), 958--1016. \textsc{doi}:\allowbreak\href{https://doi.org/10.1016/j.aim.2015.09.013}{10.1016/j.aim.2015.09.013}

\bibitem[You17a]{You17a} Matthew P.~Young, ``Weyl-Type Hybrid Subconvexity Bounds for Twisted $L$-Functions and Heegner Points on Shrinking Sets'', \textit{Journal of the European Mathematical Society} \textbf{19}:5 (2017), 1545--1576. \textsc{doi}:\allowbreak\href{https://doi.org/10.4171/JEMS/699}{10.4171/JEMS/699}

\bibitem[You17b]{You17b} Matthew P.~Young, ``Explicit Calclulations with Eisenstein Series'', preprint (2017), 37 pages. arXiv:\allowbreak\href{https://arxiv.org/abs/1710.03624}{1710.03624 [math.NT]}

\bibitem[Zag82]{Zag} Don Zagier, ``The Rankin--Selberg Method for Automorphic Functions which Are Not of Rapid Decay'', \textit{Journal of the Faculty of Science, the University of Tokyo.~Sect.~1 A, Mathematics} \textbf{28}:3 (1982), 415--437. \href{http://hdl.handle.net/2261/6300}{http://hdl.handle.net/2261/6300}

\end{thebibliography}
\end{document}